\documentclass[12pt]{amsart}
\usepackage{amsfonts}
\usepackage[margin=1.25in]{geometry}
\usepackage[alphabetic]{amsrefs} 

\makeatletter
\newtheorem*{rep@theorem}{\rep@title}
\newcommand{\newreptheorem}[2]{
\newenvironment{rep#1}[1]{
 \def\rep@title{#2 \ref{##1}}
 \begin{rep@theorem}}
 {\end{rep@theorem}}}
\makeatother

\newtheorem{theorem}{Theorem}[section]
\newreptheorem{theorem}{Theorem}
\newtheorem{lemma}[theorem]{Lemma}
\theoremstyle{definition}
\newtheorem{definition}[theorem]{Definition}
\newtheorem{example}[theorem]{Example}

\theoremstyle{remark}
\newtheorem{remark}[theorem]{Remark}
\numberwithin{equation}{section}
\theoremstyle{plain}

\newtheorem{corollary}[theorem]{Corollary}

\newtheorem{proposition}[theorem]{Proposition}

\newcommand{\R}{\mathbb{R}}

\newcommand{\sn}{\mathrm{sn}}
\DeclareMathOperator{\Ric}{Ric}
\DeclareMathOperator{\Hess}{Hess}

\begin{document}
\title{On the geometry of Riemannian manifolds with density}

\author{William Wylie}
\address{215 Carnegie Building\\
Dept. of Math, Syracuse University\\
Syracuse, NY, 13244.}
\email{wwylie@syr.edu}
\urladdr{https://wwylie.expressions.syr.edu}
\thanks{The first author was  supported by a grant from the Simons Foundation (\#355608, William Wylie).}

\author{Dmytro Yeroshkin}
\address{215 Carnegie Building\\
Dept. of Math, Syracuse University\\
Syracuse, NY, 13244.}
\email{dyeroshk@syr.edu}
\urladdr{https://dyeroshk.mysite.syr.edu}

\dedicatory{Dedicated to Frank Morgan}

\keywords{}
\begin{abstract}
We introduce a new geometric approach to a manifold equipped with a smooth density function  that takes a torsion-free affine connection, as opposed to a weighted  measure or Laplacian, as the fundamental object of study.  The connection motivates new versions of the volume and Laplacian comparison theorems that are valid for  the $1$-Bakry-\'Emery Ricci tensor, a weaker assumption than has previously been considered in the literature.  As applications we prove new generalizations of Myers' theorem and Cheng's diameter rigidity result.   We also investigate the holonomy groups of the weighted connection.  We show that they are more general than the Riemannian holonomy, but also exhibit some of the same structure.  For example,  we obtain a generalization of the de Rham splitting theorem as well as new rigidity phenomena for parallel vector fields.  A general feature of all of our rigidity results is that  warped or twisted product splittings are characterized, as opposed to the usual isometric products. 
\end{abstract}

\maketitle

\section{Introduction}

Ricci curvature for a Riemannian manifold equipped with a smooth positive density function $e^{-f}$ was first considered by Lichnerowicz \cites{Lichnerowicz70, Lichnerowicz72} and was systematically studied and vastly generalized by Bakry-\'Emery \cite{BE} and their collaborators. Their approach is to study  a weighted Laplacian and the curvature is defined to provide a Bochner formula for this operator acting on functions. This is also often viewed as the study of a manifold with smooth measure where  the measure is defined  so that  the weighted Laplacian is self-adjoint.  This approach has been extraordinarily fruitful as the Bakry-\'Emery Ricci tensors have become a fundamental concept in probability, analysis, and geometry and are important in Ricci flow, optimal transport, isoperimetric problems, and general relativity. In fact, Bakry-\'Emery's work is vastly more general, as they make sense of Bochner formulas for a much larger class of operators. Given the reach of their work, it is remarkable that their ideas in the very special case of manifolds with density has had so many applications. 

With the many applications of weighted Ricci curvature, it is also desirable to generalize other aspects of Riemannian geometry to the weighted setting. A standard introductory course in Riemannian geometry shows how the subject flows naturally from the existence of the unique Levi-Civita torsion free and compatible affine connection. In this paper we give a new approach to manifolds with density which also takes a torsion free affine connection as the fundamental object. 

Let $\nabla$ the Levi-Civita connection of a Riemannian manifold $(M,g)$. For a one-form $\alpha$ define $\nabla^{\alpha}_U V = \nabla _U V - \alpha(U)V - \alpha(V) U$. $\nabla^{\alpha}$ is a torsion free affine connection, moreover it is \emph{projectively equivalent} to $\nabla$, meaning that $\nabla^{\alpha}$ has the same geodesics, up to re-parametrization, as $\nabla$. In fact, a result of Weyl \cite{Weyl} states that any torsion-free connection projectively equivalent to $\nabla$ is of the form $\nabla^{\alpha}$ for some $\alpha$. Any connection $\widetilde{\nabla}$ has  a well defined $(1,3)$-curvature tensor and $(0,2)$-Ricci tensor given by the formulae 
\[
  R^{\widetilde{\nabla}}(X,Y) Z = \widetilde{\nabla}_X \widetilde{\nabla}_Y Z - \widetilde{\nabla}_Y \widetilde{\nabla}_ X Z - \widetilde{\nabla}_{[X,Y]} Z, 
\]
and $\mathrm{Ric}^{\widetilde{\nabla}}(Y,Z) = \mathrm{Tr}\left[X \rightarrow R^{\widetilde{\nabla}}(X,Y)Z \right]$.
 
To see the link to Bakry-\'Emery Ricci curvature, let $e^{-f}$ be a positive density function on M. The $N$-Bakry-\'Emery Ricci tensor is defined as 
\[
  \Ric_f^N = \Ric +\Hess f - \frac{df\otimes df}{N-n}, 
\]
where $N$ is a constant that is also allowed to be infinite, in which case we write $\Ric_f^{\infty} = \Ric +\Hess f$. For a manifold of dimension larger than 1, if we let $\alpha = \frac{df}{n-1}$, then we recover a Bakry-\'Emery Ricci tensor, as a simple calculation (see Proposition \ref{Prop:Curvature}) shows that 
\[
  \mathrm{Ric}^{\nabla^{\alpha}} = \mathrm{Ric}_g + \Hess f + \frac{df \otimes df}{n-1} = \Ric_f^1.
\]
In other words, the Ricci tensor of a projectively equivalent connection is exactly the $1$-Bakry-\'Emery Ricci tensor. 

The Bakry-Emery Ricci tensor has traditionally been studied for values of the parameter $n < N \leq \infty$, and this is where bounds on Bakry-Emery Ricci tensors are equivalent to the curvature dimension condition as defined by Bakry-\'Emery.    There are many recent papers on this condition, some that are most relevant to the results of this paper are  \cites{Lott, Morgan, Morgan3, MunteanuWang, WeiWylie}.  Also see chapter 18  of \cite{MorganBook} and  the references there-in.  

Note that a lower bound on $\mathrm{Ric}_f^1$ is a weaker condition than $\mathrm{Ric}_f^N$ for $n < N \leq \infty$. There is an emerging body of research on lower bounds on $\mathrm{Ric}_f^N$ when $N<n$. The first papers investigating the case $N<0$ \cites{Ohta, KolesnikovMilman} appeared almost simultaneously.  In \cite{Ohta}, Ohta extends results  involving optimal transport and lower bounds on $\mathrm{Ric}_f^N$ when $N<0$ along with extending the  Bochner inequality, eigenvalue estimates,  and the Brunn-Minkowski inequality.  In \cite{KolesnikovMilman} Kolesnikov-Milman also extend the Poincare and Brunn-Minkowski inequality for manifolds with boundary when $N<0$.  Milman also extends the Heintze-Karcher Theorem, isoperimetric inequality, and functional inequalities for $N<1$ in \cite{Milman1}.  In \cite{Ohta2} it is also established that lower bounds on $\mathrm{Ric}_f^N$, $N\leq 0$ are equivalent to curvature-dimension inequalities as defined by  Lott-Vilanni \cite{LV} and Sturm \cites{Sturm1, Sturm2}. Also see  \cites{Klartag, Milman2, OhtaTakatsu1, OhtaTakatsu2}.  For earlier related work for measures on Euclidean space see \cites{Borell, BrascampLieb}. The first results for $\mathrm{Ric}_f^1$ were proven by first author who showed generalizations of the splitting theorem \cite{WylieSplitting} and Myers' theorem \cite{WylieSec}, we will discuss these results below.   The definition of Ricci curvature for  singular torsion-free affine connections has also been investigated recently by Lott \cite{Lott2}. 

The connection $\nabla^{\alpha}$ also recovers a notion of weighted sectional curvature that was introduced by the first author.  Namely, if we let $X$ be the dual vector field to $\alpha$ coming from the metric $g$, then 
\[
  g(R^{\nabla^{\alpha}} (U, V)V, U) = \mathrm{sec}_g(U,V) + \frac{1}{2} L_X g(U,U) + g(X,U)^2.
\]
This curvature was called $\overline{\sec}_X$ in \cite{WylieSec} and was further studied in \cite{KennardWylie}. The study of weighted sectional curvature from the perspective of the connection $\nabla^{\alpha}$ will be the subject of a different paper. 

The fact the curvature quantities $\mathrm{Ric}_f^1$ and $\overline{\sec}_X$ come from the connection not only gives motivation for their study, but also introduces a number of new tools. For example, the connection $\nabla^{\alpha}$ gives a preferred re-parametrization of geodesics, which in turn can be used to define new global ``re-parametrized distance'' function $s(p,q)$ which we show plays a similar role to the distance function for comparison geometry theorems. We also show that, when $\alpha$ is closed and $M$ is orientable, $\nabla^{\alpha}$ admits a parallel volume form. This gives a natural measure which is a slightly different from the weighted measure usually used in the study of Bakry-\'Emery Ricci curvature. Combining this re-parametrized volume with the re-parametrized distance gives generalizations of the volume and Laplacian comparison theorems to $\mathrm{Ric}_f^1$ from which we obtain a new Myers' theorem along with a new diameter rigidity result. The connection also gives a weighted concept of parallelism encoded in its holonomy groups. We show these weighted holonomy groups are more general than the Riemannian ones, but also admit some similar structural properties, such as a generalization of the de Rham splitting theorem. 

In the next section we define these notions and state our main results in terms of these objects. Section~\ref{Sec:Connection} examines the basic properties of the connection $\nabla^\alpha$. Section~\ref{Sec:Comp} contains various comparison principles for manifolds with density. Section~\ref{Sec:Hol} contains a study of the holonomy group of $\nabla^\alpha$ as well as general discussion of parallel tensors.

\section{Definitions and statement of results}

\subsection{Re-parametrized distance}\label{Sec:ReparDist}

In this section we assume $\alpha$ is a closed form and write $\alpha = \frac{df}{n-1}$. The connection $\nabla^{\alpha} $ gives rise to a re-parametrization of the geodesics. We normalize these reparametrized geodesics in the following way. 
\begin{definition}
$\widetilde{\gamma}:[0,S] \rightarrow M$ is a \emph{normalized} $\alpha$-geodesic if 
\begin{enumerate}
\item $\widetilde{\gamma}$ is a re-parametrization of a minimizing unit speed Riemannian geodesic $\gamma$, and

\item $\widetilde{\gamma} = \gamma \circ s^{-1}$ where $s(t) = \int_0^t exp\left( \frac{-2f(\gamma(r))}{n-1} \right) dr$.
\end{enumerate}
\end{definition}

 For points $p, q\in M$ define the ``re-parametrized distance" between $p$ and $q$ as the infimum of the time it takes to travel from $p$ to $q$ along a normalized $\alpha$-geodesic. That is, 
\[
  s(p,q) = \inf \left\{ s: \widetilde{\gamma}(0) = p, \widetilde{\gamma}(s) = q\right\},
\]
where the infimum is taken over all normalized $\alpha$-geodesics $\widetilde{\gamma}$. 
 Assuming the metric $g$ is complete, $s$ is clearly finite and well-defined from basic properties of Riemannian geodesics. Let $s_p (\cdot) = s(p, \cdot)$. If $q$ is not a cut point to $p$, then there is a unique minimal geodesic from $p$ to $q$ and $s_p$ is smooth in a neighborhood of $q$ as can be computed by pulling the function back by the exponential map at $p$. Note that $s(p,q) \geq 0$ and is zero if and only if $p=q$ and $s(p,q) = s(q,p)$. However, $s(p,q)$ does not define a metric since it does not satisfy the triangle inequality. 
 
There is also a new natural normalization of the curvature coming from the re-parametrized geodesics. Namely, for a normalized $\alpha$-geodesic we have $\frac {d\widetilde{\gamma}}{ds} = e^{\frac{2f}{n-1}} \frac {d{\gamma}}{dr}$ so, 
\[
  \mathrm{Ric}^{\nabla^{\alpha}} \left( \frac{d\widetilde{\gamma}}{ds}, \frac{d\widetilde{\gamma}}{ds}\right) \geq (n-1)K \quad \Longleftrightarrow \quad \mathrm{Ric}_f^1\left(\frac {d{\gamma}}{dr}, \frac {d{\gamma}}{dr}\right) \geq (n-1)K e^{\frac{-4f}{n-1}}. 
\]

Our first result is the following generalization of Myers' theorem involving renormalized distance and the curvature bound. 

\begin{theorem}[Weighted Myers' Theorem]\label{Myers} 
Let $(M^n,g)$, $n>1$, be a complete Riemannian manifold and let $\alpha$ be a closed one-form, $\alpha = \frac{df}{n-1}$. Suppose that there is $K>0$ such that $\mathrm{Ric}_f^1 \geq (n-1)K e^{\frac{-4f}{n-1}}g$ then $s(p,q) \leq \frac{\pi}{\sqrt{K}}$ for all $q \in M.$ 
\end{theorem}

A corollary of the classical Myers' theorem is that the manifold must be compact. In contrast, it is possible for $s$ to be uniformly bounded on a non-compact complete Riemannian manifold. Recall that the connection $\nabla^{\alpha}$ is called \emph{geodesically complete} if all of the $\nabla^{\alpha}$-geodesics can be extended for all time. If we additionally assume completeness of the connection $\nabla^{\alpha}$ we obtain the following natural corollary. 

\begin{corollary} \label{CompMyers}
Let $(M^n,g)$, $n>1$, be a complete Riemannian manifold and let $\alpha= \frac{df}{n-1}$ be a closed one-form such that $\nabla^{\alpha}$ is geodesically complete. If $\mathrm{Ric}_f^1 \geq (n-1)K e^{\frac{-4f}{n-1}}g$ then $M$ is compact. 
\end{corollary}

\begin{remark}
Since bounded $f$ implies $\nabla^{\alpha}$ is complete, Theorem~\ref{Myers} also recovers \cite{WylieSec}*{Theorem 1.6} which states that if $\mathrm{Ric}_f^1 \geq (n-1)K$ and $f$ is bounded then $M$ is compact.
\end{remark}

The completeness of $\nabla^{\alpha}$ where $\alpha=\frac{df}{n-1}$ implies  $f$-completeness as defined in  \cite{WylieSplitting} and \cite{WoolgarWylie} (See Propositon~\ref{Prop:scomplete} below).  As such, the generalization of the Cheeger-Gromoll splitting theorem proven in \cite{WylieSplitting} can also be re-phrased in terms of the connection $\nabla^{\alpha}$. 

\begin{theorem}  \label{Splitting} \cite{WylieSplitting}*{Theorem 6.3}
Let $(M,g)$, $n>1$, be a complete Riemannian manifold and let $\alpha= \frac{df}{n-1}$ be a closed one-form such that $\nabla^{\alpha}$ is geodesically complete. If $\mathrm{Ric}_f^1 \geq 0$ and $(M,g)$ contains a line, then $(M,g)$ splits as a warped product.
\end{theorem} 

The function $s$ is also naturally related to a conformal change of metric. Let $h = e^{\frac{-4f}{n-1}} g$ then $s(p,q)$ is the smallest length in the $h$ metric of a minimal geodesic between $p$ and $q$ in the $g$ metric. As such, $s(p,q) \geq d^{h}(p,q)$. So Theorem~\ref{Myers} tells us that the diameter of the metric $h$ is less than or equal to $\frac{\pi}{\sqrt{K}}$. For this conformal diameter estimate we also obtain the following rigidity characterization. 

\begin{theorem} \label{DiamRigid}
Suppose that $(M^n,g,f)$, $n>1$, is complete and satisfies $\mathrm{Ric}_f^1 \geq (n-1) K e^{\frac{-4f}{n-1}}g$, $K>0$. If there are points $p$ and $q$ such that $d^h(p,q) = \frac{ \pi}{\sqrt{K}}$ then $(M,g)$ is a rotationally symmetric metric on the sphere. 
\end{theorem} 

Recall that Cheng \cite{Cheng} showed that a complete Riemannian manifold satisfies $\mathrm{Ric} \geq (n-1)Kg$ and $\mathrm{diam}_M = \frac{\pi}{\sqrt{K}}$ if and only if $(M,g)$ is a round sphere. We also give a complete characterization of the spaces satisfying the hypotheses of Theorem~\ref{DiamRigid} and there are examples which are not constant curvature. Theorem~\ref{DiamRigid} can thus be thought of as the analog in positive curvature to  Theorem~\ref{Splitting}. It would be interesting to know whether this result is true under the weaker condition that $s(p,q) = \frac{\pi}{\sqrt{K}}$. The main difficulty is that $s$ does not satisfy the triangle inequality. 

The re-parametrized distance $s$ also has meaning for a negative lower bound on $\mathrm{Ric}_f^1$. In fact, the results above follow from a  generalization of the Laplacian comparison theorem for the weighted Laplacian $\Delta_f = \Delta - D_{\nabla f}$ which is true for any constant $K$, see  Theorem~\ref{LapComp} below. 

\subsection{Volume Comparison}

Another insight coming from the connection is a natural volume comparison theory.  In \cite{WeiWylie} volume comparison theory was developed for  the $f$-volume, $\mathrm{Vol}_f(U) = \int_U e^{-f} dvol_g$,  for space satisfying $\mathrm{Ric}_f^{\infty} \geq Kg$ and $|f|$ is  bounded.  We  extend this theory to the weaker condition  $\mathrm{Ric}_f^1 \geq Kg$ and $|f|$ bounded.   We state only the special case here of the absolute volume comparison for $K=0$.  See Theorem~\ref{VolComp} and Corollary~\ref{Cor:VolCompBdd} below for the general statements. 

\begin{theorem} \label{VolCompBdd}
Suppose that $(M,g,f)$, $n>1$, satisfies $\mathrm{Ric}_f^1 \geq 0$ and let $f_{min}(r)$ be the minimum of $f$ on $B(p,r)$ then 
\[
  \mathrm{Vol}_f(B(p.r)) \leq \omega_nr^n e^{f(p) -2f_{min}(r)},
\]
where $\omega_nr^n$ is the volume of the ball in $n$-dimensional Euclidean space. In particular, 
\[
  \mathrm{Vol}(B(p,r)) \leq \omega_nr^n e^{2\left(f_{max}(r)- f_{min}(r)\right)},
\]
where $f_{max}(r)$ be the maximum of $f$ on $B(p,r).$
\end{theorem}  

Yang  showed that if $\mathrm{Ric}_f^{\infty} \geq 0$ and $|f|$ is  bounded, then $\mathrm{Vol}_f(B(p,r))$ has polynomial growth of degree at most $n$ \cite{Yang}.  Theorem~\ref{VolCompBdd}  extends this result to $\mathrm{Ric}_f^{1}$, and improves it,  even in the $\mathrm{Ric}_f^{\infty}$ case,  to only require a lower bound on $f$.  We also obtain the same topological application that, if $\mathrm{Ric}_f^{1}\geq 0$ and $f$ is bounded below, then $b_1(M) \leq n$ and any finitely generated subgroup of the fundamental group is polynomial growth of degree at most $n$. 

The connection $\nabla^{\alpha}$ also  provides a new approach to volume comparison theory which yields a sharp relative volume comparison that assumes no a priori bounds on $f$.  This  comes from the following observation. 

\begin{proposition} \label{Prop:ParVolForm}
Let $(M,g)$ be an orientable Riemannian manifold with Levi-Civita connection $\nabla$ and smooth one-form $\alpha$. The connection $\nabla^{\alpha}$ admits a parallel volume form if and only if $\alpha$ is closed. Moreover, if  $n>1$ and $\alpha=\frac{df}{n-1}$ the form $e^{-\frac{n+1}{n-1} f} \mathrm{dvol_g}$ is parallel with respect to $\nabla^{\alpha}$.
\end{proposition} 

For $A \subset M$, define $\mu(A) = \int_A e^{-\frac{n+1}{n-1} f } dvol_g$.  Proposition~\ref{Prop:ParVolForm}  indicates that we should consider the measure $\mu$ instead of  $\mathrm{Vol}_f$.  In fact, we will see below that the same local estimates can be used to either give bounds on $\mathrm{Vol}_f$ or $\mu$.  The measure $\mu$ arises if we change coordinates using the parametrized distance $s$ instead of the Riemannian distance function.  As such,  our volume comparison for $\mu$ will be in terms of the level sets of the  re-parametrized distance $s$ instead of the metric balls. See Theorem~\ref{VolComp} for the precise statement of this sharp relative volume comparison.

As an application of the volume comparison theorem, we obtain the following absolute volume comparison in the case $K>0$. 

\begin{theorem}  \label{ThmFiniteVolume}
Suppose that $(M^n,g,f)$ is a complete Riemannian manifold with $n>1$ supporting a function $f$ such that $\mathrm{Ric}_f^1\geq (n-1)Ke^{\frac{-4f}{n-1}}g$ for some $K>0$ then $\mu(M)<\infty$ and $\pi_1(M)$ is finite.
\end{theorem}

\begin{remark}
Theorems~\ref{Myers} and~\ref{ThmFiniteVolume} are not true for the curvature bound $\mathrm{Ric}_f^1 \geq Kg$, $K>0$.  In fact, for any Riemannian manifold $N$, $\mathbb{R} \times N$ admits a metric with density such that $\mathrm{Ric}_f^1 > g$ (see Example~\ref{Ex:K>0}).
\end{remark}

\subsection{Weighted Holonomy}

The connection $\nabla^{\alpha}$ also introduces a new concept of parallelism for manifolds with measure.  Recall that the holonomy group of a manifold equipped with a connection is the group of  linear maps of the tangent space given by parallel translation around loops with a fixed base point.  While Levi-Civita connections are characterized as the torsion free connections which have holonomy contained in $O(n)$,  on an orientable  manifold with closed 1-form $\alpha$, the holonomy of the connection $\nabla^{\alpha}$ is only contained, in general, in $SL_n(\mathbb{R})$ (see Proposition~\ref{Prop:SLn}).  This is natural if we consider the connection $\nabla^{\alpha}$ as a structure for a measure instead of a metric. While the  holonomy groups of $\nabla^{\alpha}$ are more general than the Riemannian ones, we show they also exhibit some similar rigidity phenomena. 

Recall the de Rham decomposition theorem which states that a Levi-Civita connection admits a parallel field if and only if the metric locally splits off a flat factor and, more generally, that the holonomy is reducible if and only if the metric is locally a product.  We give examples showing these results are not true for the connection $\nabla^{\alpha}$.  However,  the holonomy of $\nabla^{\alpha}$ does exhibit similar rigidity phenomena.  

The spaces in our rigidity results will be warped or twisted products instead of direct products.  Here by a \emph{twisted product} we mean a Riemannian manifold $(M, g_M)$ which is a topological product $M = B \times F$ with metric of the form $g_M= g_B + e^{2\psi}g_F$ where $g_B$ and $g_F$ are fixed metrics on $B$ and $F$ respectively and $\psi$ is an arbitrary positive function on $B\times F$. $(M,g_M)$ is a \emph{warped product} if, in addition, $\psi$ is a function depending on $B$ only.   First we state the result for parallel fields. 

\begin{theorem}\label{Thm:Parallel}
Let $\alpha = d\varphi$ be a closed $1$-form. If $(M,g)$  is complete and simply connected and admits $k$ linearly independent $\nabla^\alpha$-parallel vector fields, then $M$ splits as one of the following:
\begin{align*}
  M&= \mathbb{R}^k \times N & g_M &= g_{Eucl} + e^{2\varphi}g_N\\
  M&= \mathbb{H}^k \times N & g_M &= g_{Hyperb} + e^{2\varphi}g_N.
\end{align*}
In both cases, $\varphi$ splits as $\varphi_{\mathbb{R}^k} + \varphi_N$ (or $\varphi_{\mathbb{H}^k} + \varphi_N$), so $(M,g_M)$ can also be thought of as a warped product.
\end{theorem}

In the more general case of reducible holonomy groups we obtain a twisted product splitting. 

\begin{theorem}[Weighted de Rham decomposition theorem]\label{Thm:deRham}
A Riemannian metric  $(M,g_M)$ admits a closed one-form $\alpha=d\varphi$ such that the $\nabla^\alpha$ holonomy is reducible if and only if $(M,g_M)$ is locally isometric to a twisted product  with $g_M = g_B + e^{2\varphi} g_F$.  Moreover, if a compact manifold admits a closed one-form $\alpha$ such that the $\nabla^\alpha$ holonomy is reducible, then the universal cover is diffeomorphic to $B \times F$ and the covering metric is isometric to $g_M = g_B + e^{2\varphi} g_F$. 
\end{theorem}

\begin{remark}
We also prove a global version of the twisted product splitting in Theorem~\ref{Thm:deRham} for complete simply connected noncompact metrics $(M,g)$ satisfying an additional technical assumption (see Theorem~\ref{Thm:GlobaldeRham}). We do not know if the extra technical assumption is optimal. 
\end{remark}

\begin{remark} 
A consequence of the Riemannian de Rham theorem is that if the Riemannian holonomy is reducible then the holonomy group decomposes as a product.  This is not true for the holonomy of $\nabla^{\alpha}$ as the splitting of the group will only be block upper triangular in general for a twisted product  (see Example~\ref{Ex:Warped}).
\end{remark} 

Comparison geometry results for the weighted Ricci and sectional curvatures like the ones above prove that many topological obstructions to Riemannian metrics with curvature bounds extend to the weighted curvatures. An open question is whether the topologies that support positive curvature are the same. Namely, given a triple $(M,g,f)$ with positive weighted Ricci or sectional curvatures it is an open question whether there is always some other metric on $M$ with positive Ricci or sectional curvature. See \cites{WylieSplitting, KennardWylie} for further discussion.  

The holonomy of the connection $\nabla^\alpha$ is also related to this question by the following result. 

\begin{theorem}\label{Thm:CptCurv}
Suppose that $(M,g)$ supports a 1-form $\alpha$ such that the holonomy group of $\nabla^{\alpha}$ is compact. Then, 
\begin{enumerate}
\item If $g(R^{\nabla^{\alpha}} (X, Y)Y, X)>0$ for all orthonormal pairs $X$, $Y$, then there is a metric $\widetilde{g}$ on $M$ with $\mathrm{sec}>0$.

\item If $\mathrm{Ric}^{\nabla^{\alpha}} > 0$ then there is a metric $\widetilde{g}$ on $M$ with $\mathrm{Ric}_{\tilde{g}} > 0$. 
\end{enumerate}

Both parts also hold for non-negative, negative and non-positive curvature.
\end{theorem}

While we construct examples below showing that the holonomy group of $\nabla^{\alpha}$ need not be compact, Theorem~\ref{Thm:CptCurv} motivates the  continued study of the holonomy of $\nabla^{\alpha}$ in relation to  the study of weighted curvature bounds. For further discussion of the condition of compact $\nabla^{\alpha}$ holonomy, see subsection 5.4.


\section{Connection for Manifolds with measure }\label{Sec:Connection}

Let $(M,g)$ be a Riemannian manifold with smooth one form $\alpha$.  In this section we collect some basic facts about the weighted connection,  $\nabla^{\alpha}_X Y = \nabla_X Y - \alpha(Y)X - \alpha(X) Y$.  $\nabla^\alpha$ depends not only on $\alpha$ but on $g$ as well, however, since we will always think of the background metric $g$ as being fixed, we will not emphasize this dependency.   

It is easy to see that $\nabla^\alpha$ is a torsion free connection.  Any linear connection defines a notion of geodesics as the curves whose velocity fields are parallel along the curve. Two connections are called projectively equivalent if they have the same geodesics up to parametrization.  We call a curve an $\alpha$-geodesic if it is a geodesic for the connection $\nabla^\alpha$. We will refer to the usual geodesics for the Levi-Civita connection as the $g$-geodesics.  By a theorem of Weyl, $\nabla^{\alpha}$ is projectively equivalent to $\nabla$. For completeness we verify this fact for $\nabla^\alpha$ when $\alpha=d\varphi$ is closed,  and also fix the re-parametrization that we will utilize for comparison results.

\begin{proposition}
 If $\gamma:\R\to M$ is an $\alpha$-geodesic then the image of $\gamma$ is a $g$-geodesic, and the parametrization satisfies $|\dot{\gamma}(t)| = Ce^{2\varphi(\gamma(t))}$ for some constant $C$. 
 \end{proposition}

\begin{proof}

Let $\gamma$ be an $\alpha$-geodesic, then
\begin{align*}
  \frac{d}{dt}\langle\dot{\gamma}(t),\dot{\gamma}(t)\rangle &= 2 \langle \nabla_{\dot{\gamma}}\dot{\gamma}(t),\dot{\gamma}(t)\rangle\\
  &= 2\langle \nabla^{\alpha}_{\dot{\gamma}}\dot{\gamma}(t) + 2 d\varphi(\dot{\gamma}(t))\dot{\gamma}(t), \dot{\gamma}(t)\rangle\\
  &= 4 d\varphi(\dot{\gamma}(t)) \langle \dot{\gamma}(t),\dot{\gamma}(t)\rangle.
\end{align*}
We conclude that $2\frac{d}{dt} |\dot{\gamma}(t)| = 4 d\varphi(\dot{\gamma}(t)) |\dot{\gamma}(t)|$. Dividing by $|\dot{\gamma}(t)|$ we get a log derivative, which we solve to get $\log |\dot{\gamma}(t)| = 2 \varphi(\gamma(t)) + C$, so $|\dot{\gamma}(t)|e^{- 2 \varphi(\gamma(t))} = e^C$. The image of the geodesic is the same, since $\nabla^{\alpha}_{\dot{\gamma}} \dot{\gamma}(t) - \nabla_{\dot{\gamma}} \dot{\gamma}(t)$ is parallel to $\dot{\gamma}$.
\end{proof}

Recall that the Levi-Civita connection has the universal property that it is the unique torsion free connection which is compatible with the metric.  The next proposition shows that the weighted connection $\nabla^{\alpha}$ has  a similar universal property for a smooth manifolds equipped with a smooth measure $\mu$. 

\begin{proposition} \label{Prop:Universal} Given a Riemannian metric $(M,g)$ and a smooth measure $\mu$ there is a unique torsion free linear connection which is projectively equivalent to the Levi-Civita connection with respect to which $\mu$ is parallel. Moreover, if $\mu(A) = \int_A e^{\psi} dvol_g$, then the connection is $\nabla^{\alpha}$ where $\alpha = -\frac{1}{n+1} d \psi$. 
\end{proposition}

\begin{proof}
We use Weyl's theorem that any projectively equivalent connection is of the form $\nabla^{\alpha}_X Y = \nabla_X Y - \alpha(X) Y- \alpha(Y) X$. We can also use the fact that the Riemannian volume form is parallel with respect to the Levi-Civita connection. Then, for linearly independent fields $Y_1, \dots Y_n$ we have 
\begin{align*}
  (\nabla^{\alpha}_X e^{\psi} dvol_g)(Y_1, Y_2, \dots, Y_n )&=D_X\left(e^{\psi} dvol_g(Y_1, Y_2, \dots, Y_n )\right) \\
  & \qquad - e^{\psi} \sum_{i=1}^n (dvol_g)(Y_1, \dots, \nabla^{\alpha}_X Y_i, \dots, Y_n)\\
  &=e^{\psi} (D_X(\psi) + n \alpha(X) ) dvol_g(Y_1, Y_2, \dots, Y_n ) \\
  & \qquad +e^{\psi} \sum_{i=1}^n (dvol_g)(Y_1, \dots, \alpha(Y_i)X, \dots, Y_n)\\
  &= e^{\psi} dvol_g(Y_1, Y_2, \dots, Y_n ) \left[ D_X(\psi) + (n+1) \alpha(X) \right].
\end{align*}
Therefore, $ e^{\psi} dvol_g$ is parallel if and only if $\alpha = -\frac{1}{n+1} d \psi$.
\end{proof}

Now we turn our attention to the curvature of $\nabla^\alpha$. 

\begin{proposition} \label{Prop:Curvature}
The curvature tensor of $\nabla^\alpha$, $\alpha=d\varphi$ is
\begin{align*}
  R^{\alpha}(X,Y)Z = R(X,Y)Z &+ \Hess(\varphi)(Y,Z) X - \Hess(\varphi)(X,Z) Y\\
  &+ d\varphi(Y)d\varphi(Z) X - d\varphi(X)d\varphi(Z) Y.
\end{align*}
In particular, 
\begin{enumerate}
\item $g(R^{\alpha}(X,Y)Y, X) = \sec(X,Y) + \Hess \varphi (Y,Y) + d\varphi(Y)^2$ whenever $X,Y$ are orthonormal.

\item $\mathrm{Ric}^\alpha(Y,Z) = tr\left[X\to R^{\alpha}(X,Y)Z\right] = \mathrm{Ric}(Y,Z) + (n-1) \Hess \varphi (Y,Z) + (n-1) d\varphi(Y)d\varphi(Z)$.
\end{enumerate}
\end{proposition}

\begin{proof}
We compute the curvature tensor:
\begin{align*}
  \nabla^{\alpha}_X \nabla^{\alpha}_Y Z &= \nabla_X \nabla_Y Z - d\varphi(Z) \nabla_X Y - d(d\varphi(Z))(X) Y - d\varphi(Y) \nabla_X Z\\
  &\qquad - d(d\varphi(Y))(X) Z - d\varphi(\nabla_Y Z) X + 2 d\varphi(Y)d\varphi(Z) X\\
  &\qquad - d\varphi(X) \nabla_Y Z + d\varphi(X) d\varphi(Z) Y + d\varphi(X) d\varphi(Y) Z,
\end{align*}
similarly:
\begin{align*}
  \nabla^{\alpha}_Y \nabla^{\alpha}_X Z &= \nabla_Y \nabla_X Z - d\varphi(Z) \nabla_Y X - d(d\varphi(Z))(Y) X - d\varphi(X) \nabla_Y Z\\
  &\qquad - d(d\varphi(X))(Y) Z - d\varphi(\nabla_X Z) Y + 2 d\varphi(X)d\varphi(Z) Y\\
  &\qquad - d\varphi(Y) \nabla_X Z + d\varphi(Y) d\varphi(Z) X + d\varphi(Y) d\varphi(X) Z,
\end{align*}
and
\[
  \nabla^{\alpha}_{[X,Y]} Z = \nabla_{[X,Y]} Z - d\varphi(Z) [X,Y] - d\varphi([X,Y]) Z.
\]
Finally, we get:
\begin{align*}
  R^{\alpha}(X,Y)Z &= \nabla^{\alpha}_X \nabla^{\alpha}_Y Z - \nabla^{\alpha}_Y \nabla^{\alpha}_X Z - \nabla^{\alpha}_{[X,Y]} Z\\
  &= R(X,Y)Z + \Hess(\varphi)(Y,Z) X - \Hess(\varphi)(X,Z) Y\\
  &\qquad + d\varphi(Y)d\varphi(Z) X - d\varphi(X)d\varphi(Z) Y.
\end{align*}

This also yields (1) and (2). 
\end{proof}

To make our results in the next section easier to compare to other results for Bakry-Emery Ricci tensors, we will use the function $f = (n-1) \varphi$ so that $\alpha = \frac{df}{n-1}$.  As mentioned in the introduction, this give the formula for the Ricci tensor,  $\mathrm{Ric}^\alpha = \mathrm{Ric}_f^1$.  It also gives us the measure  $\mu(A) = \int_A e^{-\frac{n+1}{n-1} f} dvol_g$. 

Recall that a manifold with a connection is called geodesically complete if every geodesic is defined for all time. We say that  $(M,g,f)$, with $\alpha = \frac{df}{n-1}$, is $\alpha$-complete if $\nabla^\alpha$ is a geodesically complete connection. The next proposition shows that $\alpha$-completeness implies  $f$-completeness as defined in \cites{WylieSplitting, WoolgarWylie}.  Here $s$  denotes the reparametrized distance function as defined in Section~\ref{Sec:ReparDist} 

\begin{proposition}\label{Prop:scomplete}
Let $(M,g)$ be a complete and $\alpha$-complete manifold, $\alpha$ closed, then $\lim d(p,q_i)\to\infty$ implies $\lim s(p,q_i)\to\infty$.
\end{proposition}

\begin{proof}
Suppose that there exists a sequence $q_i$ such that $d(p,q_i)\to \infty$, but $s(p,q_i)$ does not go to infinity. By passing to a subsequence, we can assume that  $s(p,q_i) \leq N$ for some fixed $N$.  Let $\gamma_i:[0, T_i] \rightarrow M$ be unit speed $g$-geodesics with $\gamma_i(0)=p$ and $\gamma_i(T_i) = q_i$ such that $s(p, q_i) = \int_0^{T_i} e^{-2\varphi\left(\gamma(t)\right)} dt$.  $\dot{\gamma_i}(0)$ subconverges to a unit vector $v \in T_pM$.  Let $\gamma$ be the $g$-geodesic with $\gamma(0) = p$ and $\dot{\gamma}(0) = v$.    

By uniform convergence of geodesics, $s(p,\gamma(t)) \leq N$ for all $t>0$. Therefore, $\int_0^\infty e^{-2\varphi(\gamma(t))}dt \leq N$, so the $\alpha$-geodesic with image $\gamma$ is only defined up to some finite time.
\end{proof}


\section{Comparison Principles}\label{Sec:Comp}

\subsection{Volume element comparison}

Now we consider  comparison geometry results for $\nabla^\alpha$ on a Riemannian manifold $(M,g)$ of dimension $>1$.  In this section we will consider all of our formulas in terms of the function $f = (n-1) \varphi$.  
 
Let $p \in M$ and let $(r, \theta)$, $r>0$, $\theta \in S^n(1)$ be exponential polar coordinates (for the metric $g$) around $p$ which are defined on a maximal star shaped domain in $T_pM$ called the segment domain. Write the volume element in these coordinates as $dvol_g = \mathcal{A}(r, \theta) dr \wedge d\theta$. 
 
Let  $s_p(\cdot)$be the reparametrized distance function defined  in section 2.1 above.  Inside the segment domain,  $s_p$ has the simple formula 
\[
  s_p(r, \theta) = \int_0^r e^{\frac{-2f(t, \theta)}{n-1}} dt.
\]
Therefore, $s$ is a smooth function in the segment domain  with the property that  $\frac{\partial s}{\partial r} = e^{\frac{-2f}{n-1}}$.  We can then also take $(s, \theta)$ to be coordinates which are also valid for the entire segment domain.  We can not control the derivatives of $s$ in directions tangent to the geodesic sphere, so  the new $(s, \theta)$ coordinates are \emph{not} orthogonal as is the case for geodesic polar coordinates.    However, this is not an issue when computing volumes as 
\begin{align}
  e^{\frac{-(n+1)}{n-1} f} dvol_g &= e^{\frac{-(n+1)}{n-1} f} \mathcal{A} dr \wedge d\theta \nonumber \\
  &= e^{- f} \mathcal{A} ds \wedge d\theta. \label{VolEqn}
\end{align}
Given a minimal unit speed geodesic with $\gamma(0)=p$, the connection $\nabla^{\alpha}$  gives a natural re-parametrization of $\gamma$ in terms of the function $s$, we denote the derivative in the radial direction in terms of this parameter by $\frac{d}{ds}$.  In geodesic polar coordinates $\frac{d}{ds}$ has the expression  $\frac{d}{ds} = e^{\frac{2f}{n-1}} \frac{\partial}{\partial r}.$  We note it is not the same as $\frac{\partial}{\partial s}$ in $(s, \theta)$ coordinates. 

Recall that for a Riemannian manifold  $\frac{d}{dr} \log( \mathcal{A}) = \Delta r $, where $\Delta r$ is the Riemannian Laplacian of the distance function $r$ to the point $p$. \eqref{VolEqn} indicates we should consider the quantity 
\begin{equation}\label{VolElmS}
  \frac{d}{ds} \log \left( e^{-f} \mathcal{A}\right) = e^{\frac{2f}{n-1}} \left( \Delta r - g(\nabla f, \nabla r) \right).
\end{equation}

We thus recover the usual drift Laplacian  $\Delta_{f} u = \Delta u - g(\nabla f, \nabla u)$ considered by Bakry-\'Emery. Letting $\lambda = e^{\frac{2f}{n-1}} \Delta_f r$ we find  that $\lambda$ satisfies a familiar differential inequality in terms of the parameter $s$. 

\begin{lemma}\label{Thm:Bochner}
Let $\gamma(r)$ be a unit speed minimal geodesic with $\gamma(0) =p$.  Let $s$ be the parameter $ds = e^{\frac{-2f(\gamma(r))}{n-1}} dr$ and let $\lambda(r) = (e^{\frac{2f}{n-1}} \Delta_fr_p) (\gamma(r))$.  Then 
\[
  \frac{d \lambda}{ds} \leq  - \frac{\lambda^2}{n-1} - e^{\frac{4f}{n-1}} \mathrm{Ric}_{f}^1 \left( \frac{d \gamma}{dr}, \frac{d \gamma}{dr} \right).
\]
Moreover, if equality is achieved at a point then  at that point  $\Hess r$ has at most one non-zero eigenvalue which is of multiplicity $n-1$
\end{lemma}

\begin{proof} 
This is essentially Lemma 3.1 of \cite{WylieSplitting}.  We repeat the outline  proof here for completeness. Begin with the usual Bochner formula for functions, which says that for any $C^3$ function $u$, 
\[
  \frac12 \Delta |\nabla u|^2 = |\Hess u|^2 + \mathrm{Ric}(\nabla u, \nabla u) + g(\nabla \Delta u, \nabla u).
\]
The Bochner formula for the  $f$-Laplacian  and  $\infty$-Bakry-\'Emery Ricci curvature is given by 
\[
  \frac12 \Delta_f |\nabla u|^2 = |\Hess u|^2 + \mathrm{Ric}_f^{\infty}(\nabla u, \nabla u) + g(\nabla \Delta _fu, \nabla u).
\]
Consider this equation with $u=r_p=r$ at an interior point of a minimizing geodesic (so that $r$ is smooth in a neighborhood).  Then $|\nabla r| =1$,  so the left hand side is zero.  As $\nabla r$ is a null vector for $\Hess r$,  $\Hess r$ has at most $n-1$ non-zero eigenvalues and by Cauchy-Schwarz, $|\Hess r|^2 \geq \frac{(\Delta r)^2}{n-1}$.  This gives us the equation along $\gamma$, 
\[
  \frac{d}{dr} \left(\Delta_{f} r\right) \leq - \frac{(\Delta r)^2}{n-1} - \mathrm{Ric}_f^{\infty}\left( \frac{d \gamma}{d r},\frac{d\gamma}{d r} \right).
\]
Using this equation and completing a square gives us
\[
  \frac{d}{dr} \left( e^{\frac{2f}{n-1}} \Delta_{f} r \right) \leq -e^{\frac{2f}{n-1}}\mathrm{Ric}_{f}^1\left (\frac{d \gamma}{\partial r},\frac{d \gamma}{\partial r} \right)- e^{\frac{2f}{n-1}} \frac{(\Delta_f r)^2}{n-1}. 
\]
Since $ds = e^{\frac{-2f}{n-1}} dr$, we have the desired equation in terms of $s$. Moreover, equality in this inequality is achieved only if equality is achieved in Cauchy-Schwarz, which is equivalent to $\Hess r$ having at most one non-zero eigenvalue which is of multiplicity $n-1$. 
\end{proof}
 
Assuming the curvature bound $\mathrm{Ric}_f^1 \geq (n-1)Ke^{\frac{-4f}{n-1}}g$ we have the usual Riccati inequality 
\[
\frac{d \lambda}{ds} \leq  - \frac{\lambda^2}{n-1} - (n-1)K,
\] 
with the  caveat that it is in terms of the parameter $s$ instead of $r$. Define $\mathrm{sn}_K(s)$ be the solution to $\mathrm{sn}''_K + K \mathrm{sn}_K = 0$, $\mathrm{sn}_K(0) = 0$ and $\mathrm{sn}'_K(0) = 1$, where prime denotes derivative with respect to $s$. Define $m_K(s) = (n-1) \frac{\mathrm{sn}_K'(s)}{\mathrm{sn}_K(s)}$ so that 
\[
  m_K(s) = \begin{cases}
    (n-1) \sqrt{K} \cot\left( \sqrt{K} s\right) & K>0 \\
    \frac{n-1}{s} & K=0 \\
    (n-1)\sqrt{-K} \coth\left( \sqrt{-K} s\right) & K<0. 
  \end{cases}
\]
Then $m_K$ solves $\frac{dm_K}{ds} = \frac{-m_K^2}{n-1} - (n-1)K.$ This gives us  the following comparison estimate.

\begin{lemma} \label{MeanCurvComp}
Suppose $(M,g, f)$ be a manifold with density such that  $\mathrm{Ric}_{f}^1 \geq (n-1)Ke^{\frac{-4f}{n-1}}g$.   Let $\gamma$, $s$, and $\lambda$ be defined as in Lemma~\ref{Thm:Bochner}. Then 
\[
  \lambda(\gamma(s)) \leq m_K(s),
\]
where, when $K>0$, we assume that $s< \frac{\pi}{\sqrt{K}}$. 
\end{lemma}

\begin{proof}
Consider the quantity $\beta = \mathrm{sn}^2_K(s)\left( \lambda - m_K(s) \right)$. 
Then
\begin{align*}
  \frac{d}{ds} \left( \beta \right) &= 2 \mathrm{sn}'_K(s) \mathrm{sn}_K \left( \lambda - m_K(s) \right) + \mathrm{sn}^2_K(s)\left( \lambda' - m'_K(s) \right) \\
  &\leq \frac{\mathrm{sn}^2_K(s)}{n-1} \left( 2 \lambda m_K(s) - m^2_K(s) -\lambda^2 \right) \\
  &= -\frac{\mathrm{sn}^2_K(s)}{n-1}\left(m_K(s) - \lambda\right)^2.
\end{align*}
So we have $\beta'(s) \leq 0$. Since $sn_K(0) = 0$ the only thing we need to show to show that $\beta(s) \leq 0$ is that $\displaystyle \lim_{s \rightarrow 0} (\lambda - m_K(s))$ is bounded. To see this note that 
\[
  m_K(s) \approx \frac{n-1}{s} + O(s) \qquad s \rightarrow 0,
\]
and
\begin{align*}
  \lambda &= \left( \Delta r - g(\nabla f, \nabla r) \right) e^{\frac{2f}{n-1}} \\
  &\approx \frac{n-1}{r e^{\frac{-2f}{n-1}}} - g(\nabla f, \nabla r)e^{\frac{2f}{n-1}} + O(r) e^{\frac{2f}{n-1}} \qquad r \rightarrow 0\\
  &\approx \frac{n-1}{s} +O(1) + O(s) \qquad s \rightarrow 0. 
\end{align*}
Where in the last line we have used the fact that $\displaystyle \lim_{r \rightarrow 0} \frac{s(r)}{e^{\frac{-2f}{n-1}} r} = 1$.
\end{proof} 

This estimate gives us the following estimate for the volume element

\begin{lemma} \label{VolElement}
Suppose $(M,g,f)$ satisfies $\mathrm{Ric}_f^1 \geq (n-1)Ke^{\frac{-4f}{n-1}}g$. Let $p$ be a point in $M$ and let $\mathcal{A}$ be the volume element in geodesic polar coordinates then the function $\frac{e^{-f} \mathcal{A}}{\sn^{n-1}_K(s_p)}$ is non-increasing along any minimal geodesic with $\gamma$ with $\gamma(0) = p$. 
\end{lemma}

\begin{proof}
Define $\mathcal{A}_f = e^{-f} \mathcal{A}$.  Then from Lemma~\ref{MeanCurvComp} and \eqref{VolElmS} we have that 
\[
\frac{d}{ds} \log( \mathcal{A}_f) = e^{\frac{2f}{n-1}} \Delta_f r \leq m_K(s) = \frac{d}{ds} \log( \sn^{n-1}_K(s)).
\]
Integrating this equation between any $s_0 \leq s_1$ gives 
\[
  \log\left( \frac{\mathcal{A}_f (s_1)}{\mathcal{A}_f (s_0)}\right) \leq \log\left( \frac{ \sn^{n-1}_K(s_1)}{\sn^{n-1}_K(s_0)} \right) \quad  \Rightarrow \quad
  \frac{\mathcal{A}_f(s_1)}{\sn^{n-1}_K(s_1)}  \leq \frac{\mathcal{A}_f(s_0) }{\sn^{n-1}_K(s_0)}
\]
for all $s_0 \leq s_1$.  Note that since $ds$ is a orientation preserving change of variables along the geodesic $\gamma$, the quantity is also non-increasing in terms of the parameter $r$. 
\end{proof}

\subsection{The $n=1$ case}

Due to our normalization of the function $f$, the results of the previous section are only valid for a manifold of dimension $n>1$.    The reader might find it illuminating to consider the one dimensional case with $N=0$ where we get  ordinary differential inequalities that are similar to the equations above. To avoid further technicalities we will just consider the ODE we get when these inequalities are equalities.  Seeing the arguments in this simpler case may be helpful to the reader, but this section can also be skipped. 

Since there is no Ricci curvature in dimension $1$, we have  $\mathrm{Ric}_f^0  = \ddot{f} + (\dot{f})^2$  where $\dot{}$ denotes the derivative with respect to a parameter $r$. We consider the equation $\ddot{f} + (\dot{f})^2 = Ke^{-4f}$.  Letting $u = e^f$, this becomes 
\[
\frac{\ddot{u}}{u}=Ku^{-4} \Rightarrow \ddot{u} = Ku^{-3}.
\]
Define $\lambda = -\dot{f} e^{2f} = -\dot{u} u = -\frac{1}{2} \frac{d}{dr}(u^2)$ and the parameter $s$ such that $ds =u^{-2} dr$.  Then we have
\begin{align*}
  \frac{d\lambda}{ds}  &= -u^2 \frac{d}{dr}(u\dot{u}) \\
  &= -u^2\left(\ddot{u}u + (\dot{u})^2\right)\\
  &= -u^2\left( Ku^{-2} + (\dot{u})^2\right)\\
  &= -K - \lambda^2.
\end{align*}
This first order Riccati equation can be solved for $\lambda$.  For simplicity, we assume $K= 1, 0,$ or $-1$.  Then we obtain
\[
  \lambda(s) = \begin{cases}
    \cot\left(s + \frac{\pi}{2} - c \right) & K=1 \\
    \frac{1}{s+c} & K=0 \\
    \coth\left(s + \frac{\pi}{2} - c \right)& K=-1.
  \end{cases}
\]
On the other hand, we have $\lambda = -\frac{df}{ds}= \frac{d}{ds} \log(e^{-f})$ so we obtain 
\[
  e^{-f} = \begin{cases}
    a\sin\left(s + \frac{\pi}{2} - c \right) & K=1 \\
    as+c & K=0 \\
    a\sinh\left(s + \frac{\pi}{2} - c \right)& K=-1.
  \end{cases}
\]
where $a$ and $c$ are constants.  While we have, integrated the equations to obtain a formula for the density $e^{-f}$, we note that this is really an implicit formula as $f$ also appears in the definition of the parameter $s$. This is an issue in higher dimensions as well as our estimates are all in terms of the function $s$.

\subsection{Global comparisons}

Now we discuss the global comparison theorems that come from the local estimates of the previous section. Our first result is the Laplacian comparison theorem.
\begin{theorem}\label{LapComp}
Suppose that $(M,g)$ is complete which supports a function $f$ such that $\mathrm{Ric}_{f}^1\geq (n-1) K e^{\frac{-4f}{n-1}}g$,  then for any points $p,x \in M$
\[
(\Delta_f r_p)(x) \leq e^{\frac{-2f(x)}{n-1}} m_K(s_p(x)).
\]
\end{theorem}
 
 \begin{proof}
At the points where $r_p$ is smooth, the result follows from Lemma~\ref{MeanCurvComp}. At the points where $r_p$ is not smooth we interpret the inequality in the weak sense. If $r_p$ is not smooth at $x$, let $\gamma$ be the minimal geodesic from $p$ to $x$ such that $s(p,x) = \int_0^{r_p(x)} e^{\frac{-2(f \circ \gamma)(t)}{n-1}} dt$. For $\varepsilon>0$, let $h_{\varepsilon}(y) = \varepsilon + r(y, \gamma(\varepsilon))$. By the standard argument in the usual Ricci curvature case, $h_{\varepsilon}$ is support function for $r_p$ at $x$, meaning it is smooth in an open set containing $x$, $h_{\varepsilon}(x) = r_p(x)$, and $h_{\varepsilon}(y) \geq r_p(y)$, see, for example, Lemma 42 on page 284 of \cite{Petersen}. Moreover, by Lemma~\ref{MeanCurvComp}, where $h_{\varepsilon}$ is smooth we have 
\[
  \Delta_f h_{\varepsilon}(y) \leq e^{\frac{-2f(y)}{n-1} }m_K( s(\gamma(\varepsilon), y)).
\]
As $\varepsilon \rightarrow 0$, $s(\gamma(\varepsilon), y)$ increases to $s(p,y)$. Since $m_K$ is a decreasing function of $s$, this shows that as $\varepsilon \rightarrow 0$, $\Delta_f h_{\varepsilon}(y)$ is a decreasing function which converges to $e^{\frac{2f(y)}{n-1}} m_K(s_p(y))$. This gives the result. 
\end{proof}

Now we consider the volume comparison theorems.  Fix a point $p \in M$ and let $s_p = s(p, \cdot)$ and $r_p= r(p, \cdot)$.  We will give two versions, one for the $f$-volume, $\mathrm{Vol}_f(A) = \int_M e^{-f} dvol_g$,  of  metric annuli $A(p, r_0, r_1) = \{ x: r_0 \leq r_p(x) \leq r_1 \}$.  The comparison in this case will be in terms of the quantity 
\[
  \nu_p(n, K, r_0, r_1) = \int_{r_0}^{r_1} \int_{S^{n-1}} \sn_K^{n-1}(s_p(r, \theta)) dr d\theta.
\]
This comparison has the advantage of being for the metric annuli, however since the model measure $\nu_p$ depends on the function $s_p$, which depends on the function $f$, the model measure is not computable without further information about $f$. 

On the other hand,   the comparison space for our second  volume comparison is in terms of the traditional volume in the simply connected model space given by the formula 
\[
  v(n, K, s_0, s_1) = \int_{s_0}^{s_1} \int_{S^{n-1}} \mathrm{sn}_K^{n-1}(s) dsd \theta.
\]
However, this comparison is for the measure $\mu$, $\mu(A) = \int e^{\frac{-(n+1)}{n-1} f} dvol_g$, of the sets $C(p,s_0, s_1) = \{ x: s_0 \leq s_p(x) \leq s_1 \}$.   The $C_p$ sets,  of course, also  depend on $s_p$ and so can be quite different from annuli. In particular, since $s$ does not satisfy the triangle inequality, some care needs to be taken in using this version for traditional applications of the  volume comparison theory. 

\begin{theorem} \label{VolComp}
Suppose that $(M,g,f)$, $n>1$, satisfies $\mathrm{Ric}_f^1 \geq (n-1)Ke^{\frac{-4f}{n-1}}g$. 
\begin{enumerate}
\item  Suppose that $0 \leq r_0 \leq r_a \leq r_1$ and $0 \leq r_0 \leq r_b \leq r_1$, then 
\[
  \frac{ \mathrm{Vol}_f(A(p, r_0, r_a))}{\mathrm{Vol}_f(A(p, r_b, r_1))} \geq \frac{ \nu_p(n, K, r_0, r_a)}{\nu_p(n, K, r_b, r_1)}.
\]

\item Suppose that $0 \leq s_0 \leq s_a \leq s_1$ and $0 \leq s_0 \leq s_b \leq s_1$, then 
\[
  \frac{ \mu(C(p, s_0, s_a))}{\mu(C(p, s_b, s_1))} \geq \frac{ v(n, K, s_0, s_a)}{v(n, K, s_b, s_1)}.
\]
\end{enumerate}
\end{theorem}

\begin{proof}
Consider geodesic polar coordinates at $p$. For each $\theta \in S^{n-1}(1)$, let $\mathrm{cut}(\theta)$ to be the distance from $p$ to the cut point along the geodesic with $\gamma(0)=p$ and $\gamma'(0) = \theta$, then 
\[
  \mathrm{Vol}_f(A(p,r_0, r_a)) =\int_{S^{n-1}} d\theta \int_{\min\{\mathrm{cut}(\theta), r_0\}}^{\min\{\mathrm{cut}(\theta), r_a\}} \mathcal{A}_f(r, \theta) dr, 
\]
and 
\[
  \nu_p(n, K, r_0, r_a) =\int_{S^{n-1}} d\theta \int_{\min\{\mathrm{cut}(\theta), r_0\}}^{\min\{\mathrm{cut}(\theta), r_a\}} \sn_K^{n-1}( s(r, \theta))  dr.
\]
Then (1) follows from  Lemma~\ref{VolElement}, along with the standard proof of the relative volume comparison theorem.  See Lemma 3.2 and Theorem 3.1 of \cite{Zhu} or  Lemma 36 on page 269 of \cite{Petersen}. 
 
Similarly,  in the modified coordinates $(s, \theta)$ let $\mathrm{cut}_{s} (\theta)$ be the value of the integral $\int_0^{\mathrm{cut}(\theta)} e^{\frac{-2f(\gamma(t))}{n-1}} dt$  where $\gamma$ is the geodesic with $\gamma(0)=p$ and $\gamma'(0) = \theta$.  Then we have 
\[
  \mu(C(p,s_0, s_a)) =\int_{S^{n-1}} d\theta \int_{\min\{\mathrm{cut}_s(\theta), s_0\}}^{\min\{\mathrm{cut}_s(\theta), s_a\}} \mathcal{A}_f(s, \theta) ds, 
\]
and 
\[
  v(n, K, r_0, r_a) =\int_{S^{n-1}} d\theta \int_{s_0}^{s_a} \sn_K^{n-1}(s)  ds,
\]
and (2) follows. 
\end{proof}

The relative and absolute volume comparison for balls is the following. 

\begin{corollary}
Suppose that $(M,g,f)$ satisfies $\mathrm{Ric}_f^1 \geq (n-1)Ke^{\frac{-4f}{n-1}}g$. 
\begin{enumerate}
\item The function
\[ r \rightarrow \frac{\mathrm{Vol}_f(B(p,r))}{\nu_p(n, K, r)} \]
is non-increasing in $r$ and $\mathrm{Vol}_f(B(p,r)) \leq e^{f(p)} \nu_p(n, K, r)$ for all $r$. 

\item The function 
\[ s \rightarrow \frac{\mu(C(p, s))}{v(n,K, s)}\]
is non-increasing in $s$ and $\mu(C(p,s)) \leq e^{-\frac{n+1}{n-1}f(p)} v(n,K, s)$ for all $s$.
\end{enumerate}
\end{corollary}

As we note above, the first volume comparison has the advantage of being about metric balls, however the comparison measure $\nu_p$ still depends on the function $s$ and is thus not computable without more information about the function $f$ such as bounds or asymptotics.  In the next subsection we derive what this volume comparison says when we assume and upper and lower bound on the function $f$ on the ball of radius $r$.  On the other hand, the right hand side of the  second volume comparison is in terms of volumes in the usual model space, however it is for the measure of the level sets of $s$ instead of $r$.  We discuss applications of this volume comparison which are true with no additional assumptions about  the asymptotics of $f$ in subsection 4.5.

\subsection{The case where $f$ is bounded}

In \cite{WeiWylie} volume comparison theorems for $\mathrm{Ric}_f^{\infty} \geq (n-1)K$ are derived in terms of bounds on $f$.  In this section we describe how our volume comparison and Laplacian comparison theorem also show that similar results hold for $\mathrm{Ric}_f^1$ (and thus also for all negative values of $N$).  We consider case (1)  of Theorem~\ref{VolComp}.  For a point $p$  and $r>0$ let $f_{\max}(p,r)$ and $f_{\min}(p,r)$ be the maximum and minimum values of $f$ on the closed ball $B(p,r)$, then in geodesic polar coordinates around $p$ we have 
\begin{equation}\label{sfbound}
  e^{\frac{-2f_{\max}(p,r)}{n-1}} r \leq s_p(r, \theta) \leq e^{\frac{-2f_{\min}(p,r)}{n-1}} r.
\end{equation}
This gives us the following corollary of Theorem~\ref{VolComp}. 

\begin{corollary} \label{Cor:VolCompBdd}
Suppose that $(M,g,f)$ satisfies $\mathrm{Ric}_f^1 \geq (n-1)Ke^{\frac{-4f}{n-1}}g$.  Suppose that $0 \leq r_0 \leq r_a \leq r_1$ and $0 \leq r_0 \leq r_b \leq r_1$, and when $K>0$ assume that $\max \left\{ e^{\frac{-2f_{min}(p,r_a)}{n-1}} r_a, e^{\frac{-2f_{min}(p,r_1)}{n-1}} r_1 \right\} \leq \frac{\pi}{2 \sqrt{K}}$  then 
\[
  \frac{ \mathrm{Vol}_f(A(p, r_0, r_a))}{\mathrm{Vol}_f(A(p, r_b, r_1))} \geq \frac{ v\left(n, K, e^{\frac{-2f_{\max}(p,r_a)}{n-1}} r_0,  e^{\frac{-2f_{\max}(p,r_a)}{n-1}}r_a\right)}{v\left(n, K, e^{\frac{-2f_{\min}(p,r_1)}{n-1}}r_b,  e^{\frac{-2f_{\min}(p,r_1)}{n-1}}r_1\right)}.
\]
Moreover, $\mathrm{Vol}_f(B(p,r)) \leq e^{f(p)}  v\left(n, K,  e^{\frac{-2f_{\min}(p,r)}{n-1}} r\right)$ where,  when $K>0$ we assume $e^{\frac{-2f_{\min}(p,r)}{n-1}} r \leq \frac{\pi}{2 \sqrt{K}}$. 
\end{corollary}

\begin{proof}
When $K \leq 0$ observe that $\sn_K$ is a monotone increasing function.  Moreover, when $K>0$, the assumption that   $s_p(r, \theta) \leq e^{\frac{-2f_{min}(p,r_a)}{n-1}} r_a \leq \frac{\pi}{2\sqrt{K}}$ also implies that $\sn_K(s_p(r, \theta))$ is increasing in $r$ on $B(p, r_a)$.  Then we have 
\begin{align*}
  \nu_p(n, K, r_0, r_a) &= \int_{r_0}^{r_a} \int_{S^{n-1}} \sn_K^{n-1}\left( s_p(r, \theta) \right) dr d\theta \\
  &\geq \omega_n \int_{r_0}^{r_a}  \sn_K^{n-1}\left( e^{\frac{-2f_{\max}(p,r_a)}{n-1}} r  \right) dr\\
  &\geq \omega_n \int_{e^{\frac{-2f_{\max}(p,r_a)}{n-1}} r_0}^{e^{\frac{-2f_{\max}(p,r_a)}{n-1}} r_a}  \sn_K^{n-1}\left(u  \right) du\\
  &\geq v(n, K, e^{\frac{-2f_{\max}(p,r_a)}{n-1}} r_0,  e^{\frac{-2f_{\max}(p,r_a)}{n-1}}r_a),
\end{align*} 
where $\omega_n$ denotes the volume of the $n-1$ dimensional sphere of radius $1$ and $u$ is the linear change of variables $u = e^{\frac{-2f_{\max}(p,r_a)}{n-1}} r$. 
Similarly we obtain 
\[
  \nu_p(n,K, r_b, r_1) \leq v(n, K, e^{\frac{-2f_{\min}(p,r_1)}{n-1}}r_b,  e^{\frac{-2f_{\min}(p,r_1)}{n-1}}r_1)
\]
and
\[
  \nu_p(n, K, r) \leq v(n, K,  e^{\frac{-2f_{\min}(p,r)}{n-1}} r).
\]
Then the result follows from (1) of Theorem~\ref{VolComp}. 
\end{proof}

\begin{remark}
This statement has a particularly nice form  in the case where $K=0$ as then the comparison volumes are easy to compute, for example we have
\[
  v\left(n, 0,  e^{\frac{-2f_{\min}(p,r)}{n-1}} r\right) = \omega_n e^{-2f_{\min}(p,r)} r^n, 
\]
where $\omega_n$ is the volume of a unit ball in $\mathbb{R}^n$.  This gives us Theorem~\ref{VolCompBdd}.  
\end{remark}

\begin{remark}
Using Corollary~\ref{Cor:VolCompBdd}, one can use the same arguments as in, for example, Theorem 1.3 of \cite{WeiWylie} to show that   if $\mathrm{Ric}_f^1\geq 0$ and  $f$ admits a two sided bound then $\mathrm{Vol}_f(B(p,r))$ must grow at least linearly in $r$.  This extends Theorem 1.3 of \cite{WeiWylie}  from $\mathrm{Ric}_f^{\infty} \geq 0$ to $\mathrm{Ric}_f^1 \geq 0$. 
\end{remark}

\begin{remark}
Our volume comparison also applies to $\mathrm{Ric}^{1}_f \geq (n-1) \lambda$ when a universal two sided bound on $f$ is assumed, as we always have $\mathrm{Ric}^{1}_f \geq (n-1)Ke^{\frac{-4f}{n-1}}$. When $\lambda >0$, we let $K = \lambda e^{\frac{4f_{\min}}{n-1}}$ and, when $\lambda<0$, we let $K = \lambda e^{\frac{4f_{\max}}{n-1}}$.   
\end{remark}

We can also specialize the Laplacian comparison to the case where $f$ admits a two sided bound.  

\begin{corollary} 
Suppose that $(M,g)$ is complete which supports a function $f$ such that $\mathrm{Ric}_{f}^1\geq (n-1) K e^{\frac{-4f}{n-1}}g$,  then for any points $p,x \in M$. 
\[
  (\Delta_f r_p)(x) \leq e^{\frac{-2f_{\min}(p,r)}{n-1}} m_K\left( e^{\frac{-2f_{\max}(p,r)}{n-1}} r \right)
\]
where, when $K>0$ we assume $e^{\frac{-2f_{\max}(p,r)}{n-1}} r \leq \frac{\pi}{2\sqrt{K}}$.
\end{corollary}
 
\begin{proof}
Since functions $m_K$ are monotone decreasing, from \eqref{sfbound} we have $m_K(s(p,r)) \leq e^{\frac{-2f_{\max}(p,r)}{n-1}} r$. When $K \leq 0$,  $m_K$ is a positive function, while when $K<0$ it is non-negative on the interval $(0, \frac{\pi}{2 \sqrt{K}})$ so  the result follows from Theorem~\ref{LapComp} and again applying \eqref{sfbound}.
\end{proof}
 
When $K=0$ this gives us that 
\[
(\Delta_f r_p)(x)  \leq \frac{(n-1) e^{\frac{2(f_{\max}-f_{\min})(p,r)}{n-1}} }{r}.
\]
Compare this to the formula in \cite{WeiWylie} that states that,  under the stronger bound $\mathrm{Ric}^{\infty}_f \geq 0$,   
\[
(\Delta_f r_p)(x)  \leq \frac{n-1 + 2(f_{\max}-f_{\min})(p,r)}{r}.
\]
 
The results in this section thus show that there are uniform bounds on the $f$-volume and the $f$-Laplacian of the distance function that depend on bounds of the function $f$.  One can the use these estimates to generalize all of the results in \cite{WeiWylie} to the lower bounds on $\mathrm{Ric}_f^1$ with $|f|\leq k$. 
 
\subsection{Applications without assuming bounded $f$}
 
We now turn our attention to applications of the volume and Laplacian comparison theorems where we do not assume a priori bounds on $f$.  These results will necessarily depend on the re-parametrized distance function $s$.  Our first result is to prove Theorem~\ref{Myers}.

\begin{reptheorem}{Myers}
Suppose that $(M,g)$ is complete which supports a function $f$ such that $\mathrm{Ric}_f^1 \geq (n-1) K e^{\frac{-4f}{n-1}}g$, $K>0$ then $s(p,q) \leq \frac{\pi}{\sqrt{K}}$ for all $p,q \in M$.
\end{reptheorem}

\begin{proof}
Suppose there are points $p$ and $q$ such that $s(p,q) > \frac{\pi}{\sqrt{K}}$.  Then since the set of cut points to $p$ is closed and measure zero, by possibly changing $q$ slightly, we can assume that  $q$ is not a cut point to $p$.  From Lemma~\ref{MeanCurvComp}, along the minimal geodesic from $p$ to $q$, $\lambda \leq m_K(s)$. However, $K>0$, so, as $s \rightarrow \frac{\pi}{\sqrt{K}}$, $m_K(s)\rightarrow \-\infty$.  Since $f$ is a smooth function this implies $\Delta r \rightarrow - \infty$,  but this contradicts that fact the $r_p$ is smooth in a neighborhood of $q$.  
\end{proof}

Note that combining this with  Proposition~\ref{Prop:scomplete}  gives us Corollary~\ref{CompMyers} when the space is $\alpha$-complete.

The function $s$ is also naturally related to a conformal change of metric. Let $h = e^{\frac{-4f}{n-1}} g$ then $s(p,q)$ is the smallest length in the $h$ metric of a minimal geodesic between $p$ and $q$ in the $g$ metric. As such, $s(p,q) \geq d^{h}(p,q)$. So Theorem~\ref{Myers} tells us that the diameter of the metric $h$ is less than or equal to $\frac{\pi}{\sqrt{K}}$. 

Combining this with the volume comparison, gives us Theorem~\ref{ThmFiniteVolume}.  

\begin{reptheorem}{ThmFiniteVolume}
Suppose that $(M,g)$ is complete  and supports a function $f$ such that  $\mathrm{Ric}_f^1 \geq (n-1)K e^{\frac{-4f}{n-1}}g$, $K>0$ then $\mu(M)$ is finite and $\pi_1(M)$ is finite. 
\end{reptheorem}

\begin{proof}
We have that $C\left(p, \frac{\pi}{\sqrt{K}}\right)= M$ so by part (2) of Theorem~\ref{VolComp},
\[
\mu(M) \leq e^{-\frac{n+1}{n-1}f(p)}v\left(n,K, \frac{\pi}{\sqrt{k}}\right) < \infty.
\]
To see that $\pi_1(M)$ is finite, let $\widetilde{M}$ be the universal cover of $M$ with covering metric $\widetilde{g}$ and let $\widetilde{f}$ be the pullback of $f$ to the universal cover. Then we clearly also have $\mathrm{Ric}_{\widetilde{f}}^1 \geq (n-1)K e^{\frac{-4\widetilde{f}}{n-1}}\widetilde{g}$ on the universal cover so $\mu(\widetilde{M}) < \infty$. But, since the covering metric and $\widetilde{f}$ are invariant under the deck transformations of $\widetilde{M}$, so is the measure $\mu$, which implies  that $\pi_1(M) < \infty$. 
\end{proof}

\begin{example} \label{Ex:K>0}
Simple examples show that Theorem~\ref{Myers} and~\ref{ThmFiniteVolume} are not true for the assumption $\mathrm{Ric}_f^1 \geq Kg$. In fact,  let $(N,g_N)$ be any Riemannian  metric with bounded curvature and consider the metric  $g = dr^2 + e^{2r}g_N$ on $\mathbb{R} \times N$.  Then the Ricci tensor of $g$ is given by 
\begin{align*}
  \mathrm{Ric}\left( \frac{\partial}{\partial r} , \frac{\partial}{\partial r} \right) & = -(n-1) \\
  \mathrm{Ric}\left( \frac{\partial}{\partial r} , U \right) &= 0 \\
  \mathrm{Ric}\left( V , U \right) &=  -(n-1) +  e^{2r} \mathrm{Ric}^{g_N}(U,V),
\end{align*}
where $U, V$ are fields on $N$.  Consider the function $f= Ar$ for some constant $A$, then 
\begin{align*}
  \left( \Hess f + \frac{df \otimes df}{n-1} \right)\left( \frac{\partial}{\partial r} , \frac{\partial}{\partial r} \right) & = \frac{A^2}{n-1} \\
  \left( \Hess f + \frac{df \otimes df}{n-1} \right)\left( \frac{\partial}{\partial r} , U \right) & = 0 \\
  \left( \Hess f + \frac{df \otimes df}{n-1} \right)\left( U,V\right) & = A e^{2r} g_N(U,V).  \\
\end{align*}
Thus we can see that taking $A>0$ to be sufficiently large will give  $(M,g,f)$ with $\mathrm{Ric}_f^1 \geq Kg$ on $N \times \mathbb{R}$.  Moreover, the geodesic obtained by letting $r \rightarrow -\infty$ will have $s \rightarrow \infty$.
\end{example}

\subsection{Rigidity}
Now we discuss rigidity in the comparison estimates in the previous section. The basic idea is to understand equality in the estimates in section 4.1.  In this section we will let $\dot{}$ denote the derivative with respect to $r$, which is a unit speed parametrization of a minimal geodesic. 

\begin{lemma} \label{RigidLemma}
Suppose $(M,g, f)$ is a manifold with density and $\gamma$ is a $g$-geodesic such that $\mathrm{Ric}_f^1(\dot{\gamma}, \dot{\gamma}) \geq (n-1)K e^{\frac{-4f}{n-1}}$.  Moreover, suppose that for some $r_0$, $ \lambda(\gamma(r_0)) = m_K(s(r_0)), $
then 
\begin{enumerate}
\item $\mathrm{Ric}_f^1(\dot{\gamma}(r), \dot{\gamma}(r)) = (n-1)K e^{\frac{-4(f \circ \gamma)(r)}{n-1}}$ and $\lambda(\gamma(r)) = m_K(s(r))$ for all $0 < r \leq r_0$, 

\item $g_{\gamma(r)}= dr^2 + e^{\frac{2(f \circ \gamma(r))}{n-1}} \mathrm{sn}^2_K(s(r)) g_{S^{n-1}}$ for all $0 < r \leq r_0$. 
\end{enumerate}
\end{lemma}

\begin{remark}
Conversely, for any function $f(r, \theta)$ defined on $(a,b) \times S^{n-1}$, if we define $g = dr^2 + e^{\frac{2f}{n-1}} \mathrm{sn}^2_K(s) g_{S^{n-1}}$ where $s(t, \theta) = \int_0^r e^{\frac{-2f(t, \theta)}{n-1}}dt$, then we will have 
 $\mathrm{Ric}_f^1(\frac{\partial}{\partial r},\frac{\partial}{\partial r}) = (n-1)K e^{\frac{-4(f \circ \gamma)(r)}{n-1}}$. 
\end{remark}
 
\begin{proof}
(1) clearly follows from Lemma~\ref{Thm:Bochner} since equality at $r_0$ implies that the derivatives are equal for all $r<r_0$. Then from Lemma~\ref{Thm:Bochner} we also have $\Hess r = A(r )g_r$ for all $0 < r \leq r_0$. Then we have 
\[
\lambda = ((n-1)A(r) - g(\nabla f, \nabla r))e^{\frac{2(f \circ \gamma)(r)}{n-1} }= m_K(s(r)).
\]
So we have $A(r) = e^{\frac{-2(f \circ \gamma)(r)}{n-1}}\frac{\mathrm{sn}_K'}{\mathrm{sn}_K} (s(r)) + \frac{g(\nabla f, \nabla r)}{n-1}$. 

Let $\{ \frac{\partial}{\partial r}, \frac{\partial}{\partial \theta^i}\}_{i=1}^{n-1}$ be a geodesic polar coordinates around $p$. Then the standard equation $L_{\frac{\partial}{\partial r}} g = 2 \Hess r$, in these coordinates,  becomes  
\[
  \frac{\partial}{\partial r} g_{ij} = 2 (S^k_i) g_{kj},
\]
where $S^k_i$ are the components of the matrix representing $\Hess r$. In our case $S^k_i = A(r) \delta_i^k$ so we have that $g_{ik} = 0$ for $i \neq k$ and 
\[
  \frac{\partial}{\partial r} g_{ii} = 2 (A) g_{ii} = \left( \frac{2}{n-1} \frac{\partial f}{\partial r}+ 2e^{\frac{-2f}{n-1}}\frac{\mathrm{sn}_K'}{\mathrm{sn}_K} (s) \right) g_{ii}.
\]
Which implies that 
\[
  \frac{\partial}{\partial r} \left( \log(g_{ii}) - \frac{2f}{n-1} \right) = 2e^{\frac{-2f}{n-1}} \frac{\mathrm{sn}_K'}{\mathrm{sn}_K} (s).
\]
In terms of the parameter $\frac{d}{ds} = e^{\frac{2f}{n-1}}  \frac{\partial}{\partial r}$ this gives us
\[
\frac{d}{d s} \left( \log(g_{ii}) - \frac{2f}{n-1} \right) =2 \frac{\mathrm{sn}_K'}{\mathrm{sn}_K} (s).
\]
Integrating with respect to $s$ gives 
\begin{align*}
  \frac{ e^{ - \frac{2f(s)}{n-1}} g_{ii}(s) }{e^{ - \frac{2f(s^{-1} (\varepsilon))}{n-1}} g_{ii}(s^{-1} (\varepsilon))}&= \frac{\mathrm{sn}^2_K(s)}{\mathrm{sn}^2_K(\varepsilon)} \\
  \frac{e^{ - \frac{2f(s)}{n-1}} g_{ii}(s)}{ \mathrm{sn}^2_K(s)} &= \lim_{\varepsilon \rightarrow 0} \frac{e^{ - \frac{2f(s^{-1} (\varepsilon))}{n-1}} g_{ii}(s^{-1} (\varepsilon))}{ \mathrm{sn}^2_K(\varepsilon)} = 1.
\end{align*}
\end{proof}

There is also more rigidity given when the function $s(p,q)$ is equal to the distance in the conformal metric $h$. 

\begin{proposition}
Let $p$ and $q$ be points on $(M,g,f)$, such that $s(p,q) = d^{h}(p,q)$ and let $\gamma(r)$ be the minimal $g$-geodesic from $p$ to $q$ such that $s(p,q) = \int_0^r e^{\frac{-2(f\circ \gamma)(t)}{n-1}} dt$, then $\nabla f$ is parallel to $\frac{d\gamma}{dr}$. 
\end{proposition}

\begin{proof}
If $s(p,q) = d^h(p,q) = \mathrm{length}^h(\gamma)$, then $\gamma$ is also a minimal geodesic in the $h$ metric. In particular, $\nabla^h_{\frac{d\gamma}{d s} }\frac{d\gamma}{d s} = 0$.  Applying the  formula for connection of $h$ in terms of $g$ gives us
\begin{align*}
0 &= \nabla^h_{\frac{d\gamma}{ds} }\frac{d \gamma}{d s} \\
&= \nabla^g_{\frac{d\gamma}{ds} }\frac{d \gamma}{d s} - \frac{4}{n-1} g\left(\frac{d\gamma}{ds}, \nabla f\right) \frac{d\gamma}{ds}+\frac{ 2}{n-1} g\left(\frac{d\gamma}{ds}, \frac{d\gamma}{ds}\right) \nabla f \\
&= \frac{2e^{\frac{4f}{n-1}}}{n-1} \left( - g\left(\frac{d\gamma}{dr}, \nabla f\right) \frac{d\gamma}{dr}+ \nabla f \right).
\end{align*}
Which is true if and only if $\nabla f$ is parallel to $\frac{d\gamma}{dr}$.
\end{proof}

Using these lemmas, we can now obtain a rigidity result for Theorem~\ref{Myers}. Here we do encounter the problem that $s$ does not satisfy the triangle inequality. To get around this issue we instead consider rigidity in the the diameter estimate for the conformal metric $h = e^{\frac{-4f}{n-1}} g$. From $s(p,q) \geq d^{h}(p,q)$ and the triangle inequality for the $h$-metric we have
\[
s(p,x) + s(q,x) \geq d^h(p,x) + d^h(q, x) \geq d^{h}(p,q).
\]

\begin{theorem}
Suppose that $(M,g,f)$ is a complete manifold with density and satisfies $\mathrm{Ric}_f^1 \geq (n-1) K e^{\frac{-4f}{n-1}}g$, $K>0$. Then there are points $p$ and $q$ such that $d^h(p,q) = \frac{ \pi}{\sqrt{K}}$ if and only if $g$ is a rotationally symmetric metric on the sphere of the form 
\[
g = dr^2 + \frac{ e^{\frac{2(f(r))}{n-1}} \sin^2(\sqrt{K}s(r)) }{K}g_{S^{n-1}}, \qquad 0 \leq r \leq D
\]
where $f$ is a function of $r$ such that $\int_0^D e^{\frac{-2f(t)}{n-1}} = \frac{\pi}{\sqrt{K}}.$ 
\end{theorem}

\begin{proof}
Let $r_p$ and $r_q$ be the distance functions to $p$ and $q$ respectively, then for any point $x \in M$, by Theorem~\ref{LapComp},  we have 
\[
 \Delta_f (r_p + r_q)(x) \leq (n-1) \sqrt{K} e^{\frac{2f(x)}{n-1}} \left( \cot(\sqrt{K} s_p(x)) + \cot(\sqrt{K} s_q(x)) \right).
\]
We also have $s_p(x) + s_q(x) \geq d^h(p,q) = \frac{ \pi}{\sqrt{K}}$, so that 
\[
 \cot(\sqrt{K} s_q(x)) \leq \cot( \pi - \sqrt{K} s_p(x)) = - \cot(\sqrt{K} s_p(x)).
\]
Thus,  $\Delta_f (r_p + r_q) \leq 0$. 

On the other hand, by the triangle inequality, $r_p(x) + r_q(x) \geq d(p,q)$, with equality if and only if $x$ is on a minimal geodesic from $p$ to $q$.  Thus $r_p+r_q$ always attains a local minimum.  By the minimum principle, this shows that $r_p(x) + r_q(x) = d(p,q)$ for all $x$ and  all geodesics starting at $p$ in $M$ are minimizing and end at $q$. 

We also have $\Delta_f (r_p + r_q) =0$ everywhere so we have equalities in all of the inequalities above. Firstly this tells us that $s_p(x) + s_q(x) = d^h(p,q) = \frac{ \pi}{\sqrt{K}}$ so that $s_p(x) = d(p,x)$ and thus $\nabla f$ must be parallel to every minimal geodesic emanating from $p$ so that $f$ must be a function of $r$. Secondly we also have equality  as in Lemma~\ref{RigidLemma} along all geodesics starting at $p$ which gives us a metric of the form 
\[
g = dr^2 + \frac{e^{\frac{2(f(r))}{n-1}} \sin^2(\sqrt{K} s(r))}{K} g_{S^{n-1}}, \qquad 0 \leq r \leq D = d(p,q).
\]
Conversely we can see that any such metric will satisfy the hypotheses of the lemma as long as $\int_0^D e^{\frac{-2f(t)}{n-1}} = \frac{\pi}{\sqrt{K}}.$
\end{proof}

For rigidity in Theorem~\ref{ThmFiniteVolume} when $K>0$ we have the following. 

\begin{theorem} \label{VolRigid1}
Suppose that $(M,g)$ is complete which supports a function $f$ such that $\mathrm{Ric}_f^1 \geq K e^{\frac{-4f}{n-1}}g$, $K>0$, and there is a point $p$ such that 
\[
\mu(M) = e^{-\frac{n+1}{n-1} f(p)} v\left(n, K, \frac{\pi}{\sqrt{K}} \right),
\]
then either
\begin{enumerate} 
\item $M$ is diffeomorphic to either $S^n$ of diameter $D$ and the metric is of the form 
\[
  g = dr^2 + e^{\frac{2f}{n-1}}\frac { \sin^2( \sqrt{K} s) } { K} g_{S^{n-1}} \qquad r \in [0, D]
\]

\item or, $M$ is diffeomorphic to $\mathbb{R}^n$ with metric of the form 
\[
  g = dr^2 + e^{\frac{2f}{n-1}}\frac { \sin^2( \sqrt{K} s) } { K} g_{S^{n-1}} \qquad r \in [0, \infty).
\]
\end{enumerate}
\end{theorem}

\begin{proof}
We have equality in the volume comparison
\[
  \mu(C(p, \frac{\pi}{\sqrt{K}})) =e^{-\frac{n+1}{n-1} f(p)} v\left(n, K, \frac{\pi}{\sqrt{K}} \right),
\]
which implies that for every geodesic with $\gamma(0)=p$ must have $\int_0^{\mathrm{cut}(\gamma)} e^{\frac{-2(f \circ \gamma)(t)}{n-1}} dt = \frac{\pi}{\sqrt{K}}$ where $\mathrm{cut}(\gamma)$ is defined to be the number such that $\gamma(\mathrm{cut}(\gamma))$ is a cut point to $p$ and is $\infty$ if there is no cut point along $\gamma$. Then we also have equality in the mean curvature comparison so we must have $g = dr^2 + e^{\frac{2f}{n-1}}\frac { \sin( \sqrt{K} s) } { K} g_{S^{n-1}}$ in the segment domain, where $f$ and $s$ are functions of $r$ and $\theta$. 

We now have to show that the metric must be a metric on the sphere or Euclidean space. This follows because $g$ is of the form $g = dr^2 + \varphi^2(r, x) g_{S^{n-1}}$ in the segment domain. First note that if the injectivity radius is infinite then the exponential map is a diffeomorphism and $M$ is diffeomorphic to $\mathbb{R}^n$ and we are in case (2). On the other hand, suppose that the injectivity radius at $p$ is $D<\infty$ and let $\gamma$ be a minimizing geodesic with $\gamma(0) = p$ such that $\mathrm{cut}(\gamma)=D$ and let $q = \gamma(D)$. Then from the discussion above we have $s \rightarrow \frac{\pi}{\sqrt{K}}$ at $\gamma(D)$ and so the function $\varphi (r, x) \rightarrow 0$ as $r \rightarrow D$. On the other hand, the Jacobi fields of the metric $g$ along a radial geodesic are of the form $J = \varphi E$, where $E$ is a fixed field in the geodesic sphere. This shows that the index of $q$ as a conjugate point to $p$ along $\gamma$ is $(n-1)$. 
 
Now consider the set $S$ of vectors $w \in T_pM$ such that $exp_p(w) = q$. We want to show that if $|w|=D$ then $w \in S$. To see this, let $\theta(t)$ be a curve in the sphere of radius $D$ with $\theta(t) \in S$, then consider the geodesic variation $v(r, t) = exp_p ( \frac{\theta(t) r}{D})$ the derivative with respect to $t$ at $0$ is a Jacobi field along $\gamma$ which vanishes at $\gamma(0)$, since the index is $(n-1)$ it must be a proper Jacobi field, i.e. $\frac{\partial}{\partial t} v(t, D) = 0$, so $v(t,D) = v(0, D)= q$. This shows that the set $S$ is both open and closed in the sphere of radius $D$, and thus must be the whole sphere. It follows that every geodesic $\gamma$ must have $\mathrm{cut}(\gamma)= D$ and that $\gamma(D) = q$. This gives us case (1). 
\end{proof}

When $K \leq 0$ we have a similar local rigidity result for the volume.  

\begin{theorem} \label{VolRigid2}
Suppose that $(M,g)$ is complete and satisfies $\mathrm{Ric}_f^1 \geq Ke^{\frac{-4f}{n-1}}g$ and there is a $p$ and $S$ such that $\mu(C(p,S)) = e^{-\frac{n+1}{n-1} f(p)} v\left(n, K, S \right)$, 
then $C(p,S)$ is diffeomorphic to an open disk with twisted product metric 
\[
g = dr^2 + e^{\frac{2(f(r, \theta))}{n-1}} \mathrm{sn}^2_K(s(r, \theta) ) g_{S^{n-1}}.
\]
\end{theorem}

\begin{proof}
The argument is very similar to the above, we must have that there are no cut points to $p$ in $C(p,S)$ which gives the topological structure, and then the form of the metric comes from the equality in the mean curvature comparison. 
\end{proof} 


\section{$\alpha$-Holonomy}\label{Sec:Hol}

In this section we investigate the possible holonomy groups of $\nabla^\alpha$, which we will call $\alpha$-holonomy and denote by $\mathrm{Hol}^\alpha(M)$. Recall that for a linear connection on a connected manifold, the holonomy group at $p$ is the group of invertible linear maps $h: T_pM \rightarrow T_pM$ given by parallel translation around some $C^1$-piecewise loop with basepoint $p$. Furthermore, one can show that the holonomy group is in fact a Lie group, and if $M$ is simply connected, $\mathrm{Hol}^\nabla(M)$ is connected. Changing the point $p$ conjugates the holonomy group, so holonomy groups at different points are all isomorphic and we can talk about the holonomy of the connection.  For this and other background on holonomy groups of linear connections we refer the reader to \cites{Besse, KobayashiNomizu, Petersen} and the reference there-in.

The structure of the holonomy groups is related to the existence of parallel structures by the following ``fundamental principle'' (see \cite{Besse}*{10.19}). 

\begin{proposition} \label{Prop:FundPrinc1}
Let $\nabla$ be a linear connection on a connected manifold. Fix non-negative integers $r$ and $s$, then the following are equivalent 
\begin{enumerate}
\item There exists a tensor field of type $(r,s)$ which is invariant under parallel transport.

\item There exists a tensor field of type $(r,s)$ which has zero covariant derivative.

\item There exists a point $p \in M$ and, on $T_pM$, a tensor $\mathcal{T}_p$ of type $(r,s)$ which is invariant under the holonomy group of the connection at $p$. 
\end{enumerate}
\end{proposition}

In particular, since the metric is always parallel with respect to its Levi-Civita connection,  the holonomy group of a Levi-Civita connection is isomorphic to a subgroup of $O(n)$. Berger \cite{Berger} classified  the groups that can arise as the holonomy of a Levi-Civita connection on a simply connected manifold and the classification is quite restrictive.  On the other hand, Hano and Ozeki \cite{HanoOzeki} have shown for a general connection that, as long as the tangent bundle can be reduced to a given closed linear subgroup $G$ (a topological condition), then there is a linear connection with holonomy $G$. The groups that can arise as the holonomy group of a torsion free connection have also been classified by Merkulov-Schwachh\"ofer \cite{MeSch} and Bryant \cite{Bryant}.   We are interested in the possibilities for $\alpha$-holonomy. Our first consideration is the existence of a parallel $n$-form. 

\begin{proposition} \label{Prop:SLn}
Let $\alpha$ be a closed one-form on a Riemannian manifold $(M,g)$, then  $\mathrm{Hol}^\alpha(g)$ is isomorphic to a subgroup of the group $\{ A \in GL(n, \mathbb{R}) : \mathrm{det}(A) = \pm 1 \}$. Moreover, $(M,g)$ is orientable if and only if $\mathrm{Hol}^\alpha(g)$ is isomorphic to a subgroup of $\mathrm{SL}_n(\mathbb{R})$. 
\end{proposition}

\begin{proof}
First we consider the orientable case. By Proposition~\ref{Prop:Universal} if $(M,g)$ is orientable, then there is an $n$-form which is parallel with respect to $\nabla^\alpha$. By Proposition~\ref{Prop:FundPrinc1} this implies that there is a volume form on $T_pM$ which is invariant under $\mathrm{Hol}_p^\alpha$, showing that $\mathrm{Hol}_p^\alpha$ is isomorphic to a subgroup of $\mathrm{SL}_n(\mathbb{R})$. Conversely, if $\mathrm{Hol}_p^\alpha$ is isomorphic to a subgroup of $\mathrm{SL}_n(\mathbb{R})$, then by Proposition~\ref{Prop:FundPrinc1}, $(M,g)$ admits a parallel, non-degenerate $n$-form. Thus, $(M,g)$ must be orientable. 

In the non-orientable case, consider the orientable double cover, $(\widetilde{M}, \widetilde{g})$ with the 1-form $\widetilde{\alpha}$ being the lift of $\alpha$. Let $h \in \mathrm{Hol}^\alpha_p(g)$ which is given by parallel translation around a curve $\sigma$. Then the curve which is concatenation of $\sigma$ with itself lifts to a loop in $\widetilde{M}$. Since $\widetilde{\alpha}$-parallel translation in the cover is the same as $\alpha$-parallel translation in $M$, this shows that $\mathrm{det}(h^2) = 1$, which implies that $\mathrm{det}(h) = \pm 1$. 
\end{proof}

\begin{remark}
Conversely, by  Propositions~\ref{Prop:Universal} and~\ref{Prop:FundPrinc1}, if $\alpha$ is not a closed 1-form, then $\mathrm{Hol}^\alpha(M)$ does not lie in $\{ A \in GL(n, \mathbb{R}) : \mathrm{det}(A) = \pm 1 \}$.
\end{remark}

In the next section we construct examples showing that $\mathrm{Hol}^\alpha$ can be all of $\mathrm{SL}_n(\mathbb{R})$. 

\subsection{Examples}

In this section we try to build intuition about $\alpha$-Holonomy by collecting examples. 
 
We consider the ansatz of a twisted product. Let $M = B \times F$ with metric of the form $g_M = g_B + e^{2\psi} g_F$ where $\psi:B \times F \rightarrow \mathbb{R}$. Let $\varphi:B \times F \rightarrow \mathbb{R}$ be an arbitrary function, with $\alpha = d\varphi$. Let $X, Y, Z$ be fields on $B$ and $U,V, W$ be fields on $F$. The connection is 
\begin{align*}
  \nabla_X^M Y &= \nabla_X^B Y \\
  \nabla_X^M U &= \nabla_U^M X = d\psi(X) U \\
  \nabla_V^M U &= \nabla_V^F U + d\psi(U)V + d\psi(V) U - e^{2\psi}g_F(U,V) \nabla \psi.
\end{align*}
So $\nabla^\alpha$ is given by 
\begin{align}
\label{twistprod1}
\begin{split}
  \nabla_X^\alpha Y &= \nabla_X^B Y - d\varphi(X) Y - d\varphi(Y) X \\
  \nabla_X^\alpha U &= \nabla_U^\alpha X = (d\psi(X) - d\varphi(X)) U - d\varphi(U) X \\
  \nabla_V^\alpha U &= \nabla_V^F U + (d\psi(U) - d\varphi(U) )V\\
  & \qquad\qquad+ (d\psi(V)- d\varphi(V)) U - e^{2\psi}g_F(U,V) \nabla \psi .
\end{split}
\end{align}
Let $\sigma(t)$ be a curve in $M$ and write $\sigma(t) = (\sigma_1(t), \sigma_2(t))$ where $\sigma_1$ and $\sigma_2$ are curves in $B$ and $F$ respectively. We use $\dot{ }$ to denote derivative in the $t$ direction. 
 
Let $P(t) = a(t) X(t) + b(t) U(t)$ where $X$ is a vector field on $B$ and $U$ is a vector field on $F$. Then using the equations above we have 
\begin{align}
\label{Peqn} 
\begin{split}
\nabla_{\dot{\sigma}}^\alpha P   &= \dot{a} X + a\left( \nabla_{\dot{\sigma_1}} X - d\varphi(X) \dot{\sigma}_1 - d\varphi(\dot{\sigma}_1) X \right) + a(d\psi - d\varphi)(X) \dot{\sigma}_2  \\
 &\quad - ad\varphi( \dot{\sigma}_2) X + b(d\psi-d\varphi)(\dot{\sigma}_1) U - b d\varphi(U) \dot{\sigma}_1  + \dot{b} U   \\
 &\quad + b \left( \nabla_{\dot{\sigma}_2}^F U + (d\psi - d\varphi)(U) \dot{\sigma}_2 + (d\psi - d\varphi)(\dot{\sigma}_2) U + e^{2\psi} g_F(U, \dot{\sigma}_2) \nabla \psi \right).
\end{split}
\end{align}
 
Recall the  de Rham splitting theorem which states that if the holonomy of a Levi-Civita connection is reducible, then the metric splits as a product.   This result is not true for $\alpha$-holonomy.  To see this, recall the correspondence  between holonomy invariant subspaces  a tangent space and invariant distributions under parallel translation.  See \cite{Besse}*{10.21}. 

\begin{proposition} \label{Prop:FundPrinc2}
Let $\nabla$ be a linear connection on a connected manifold $M^n$. Let $k$ be an integer between $1$ and $n-1$, then the following are equivalent 
\begin{enumerate}
\item There exists a $k$-dimensional distribution which is preserved by parallel transport.

\item The holonomy group leaves invariant a subspace of dimension $k$. 
\end{enumerate}
Moreover, if the connection is torsion free then the distribution is integrable. 
\end{proposition}

Using this and the equations above we can see that there are twisted product metrics that have reducible $\alpha$-holonomy. 

\begin{proposition} \label{Prop:TwistedProd}
Suppose $(M,g_M)$ is a twisted product as above. Assume further that $\varphi$ and $\psi$ differ by a function on $F$. Then the tangent space on $B$ is invariant under the $\alpha$-Holonomy. 
\end{proposition}

\begin{proof}
Assuming that $\varphi = \psi + \varphi_2$ where $\varphi_2 : F \rightarrow \mathbb{R}$ and taking $U = 0$ in (\ref{Peqn}) gives
\[
  \nabla_{\dot{\sigma}}^\alpha P = \dot{a} X + a\left( \nabla_{\dot{\sigma_1}} X - d\varphi(X) \dot{\sigma}_1 - d\varphi(\dot{\sigma}_1) X- d\varphi( \dot{\sigma}_2) X \right ).
\]
Let $a = e^\varphi$ and $X$ be a field that satisfies $\nabla_{\dot{\sigma_1}} X = d\varphi(X) \dot{\sigma}_1$, then we obtain a solution, so we have that the vector field $P = e^\varphi X$ will be a parallel field. This shows that if $v$ is a vector in the $B$ factor, then its parallel translate will also be in the $B$ factor for all time. In particular, the tangent space of $B$ will be preserved by the holonomy. 
\end{proof}

On the other hand, we will prove a generalization of the de Rham splitting theorem below showing that if the $\alpha$-holonomy of a manifold is reducible, then the metric splits as a twisted product.  If we assume further that the metric a warped product we get some more explicit formulae for $\alpha$-holonomy. 

\begin{proposition}\label{Prop:Twisted2}
Suppose $(M,g_M)$ is a warped product.  That is $g_M$ is a twisted product as above with  $\psi$ is a function of $B$ only, and that $\varphi = \psi + \varphi_2$ . Let $h$ be the element of $\alpha$-holonomy coming from parallel translation around a loop $\sigma$. In terms of the splitting of the tangent space, $T_{\sigma(0)}M = T_{\sigma_1(0)}B + T_{\sigma_2(0)} F$ the matrix of $h$ is of the form
\[
h = \begin{pmatrix} A & C \\ 0 & h_2 \end{pmatrix},
\]
where $h_2$ is the element of the $\alpha_2$-holonomy on $F$ generated by the loop $\sigma_2$, where $\alpha_2 = d\varphi_2$.
\end{proposition}

\begin{proof} 
Consider (\ref{Peqn}) with $a = e^\varphi$, $b=1$. Then we obtain, 
\begin{align*}
  \nabla_{\dot{\sigma}}^\alpha P &= e^\varphi \left( \nabla_{\dot{\sigma_1}} X - d\psi(X) \dot{\sigma}_1\right) - d\varphi_2(U) \dot{\sigma}_1 + e^{2\psi} g_F(U, \dot{\sigma}_2) \nabla \psi \\
  &\qquad + \left( \nabla_{\dot{\sigma}_2}^F U - d\varphi_2(U) \dot{\sigma}_2 - d\varphi_2(\dot{\sigma}_2) U \right) .
\end{align*}
Let $U$ be an $\alpha_2$-parallel field along $\sigma_2$, let $X=X_U$ be the solution to 
\begin{align}
\label{X_Ueqn}  0= e^\varphi \left( \nabla_{\dot{\sigma_1}} X - d\psi(X) \dot{\sigma}_1\right) - d\varphi_2(U) \dot{\sigma}_1 + e^{2\psi} g_F(U, \dot{\sigma}_2) \nabla \psi.
\end{align}
Which, after $U$ has been fixed, is an in-homogeneous first order system of equations for $X(t)$. Then we have that the $\alpha$-parallel fields are all of the form $U + e^\varphi X_U$. 
 
Putting this together in terms of the $\alpha$-holonomy around a loop $\sigma$ and the splitting of the tangent space, $T_{\sigma(0)}M = T_{\sigma_1(0)}B + T_{\sigma_2(0)} F$ we have that the matrix for the element $h$ of $\mathrm{Hol}^\alpha(M,g)$ coming from parallel translation around $\sigma$ is of the form
\[
  h = \begin{pmatrix}
    A & C \\
    0 & h_2
  \end{pmatrix},
\]
where $h_2$ is the element of the $\alpha_2$-holonomy on $F$ generated by the loop $\sigma_2$. $A$ is the matrix given by solving the system $\nabla_{\dot{\sigma_1}} X = d\psi(X) \dot{\sigma}_1$ and $C$ is the matrix given by solving for $X_U$ in (\ref{X_Ueqn}).
\end{proof} 
 
\begin{example}\label{Ex:Warped}
Consider a product metric where $\varphi$ is a function on $F$. This is the case where $\psi = 0$ and $\varphi = \varphi_2$. Then, since $d\psi= 0$, $A$ will just be $h_1$, the holonomy element generated by parallel translation around $\sigma_1$. So we have
\[
  h = \begin{pmatrix}
    h_1 & C \\
    0 & h_2
  \end{pmatrix},
\]
where, by (\ref{X_Ueqn}),  $C$ is given by the solutions to $ \nabla_{\dot{\sigma_1}} X = e^{-\varphi_2} d\varphi_2(U) \dot{\sigma}_1$. 

To make this more concrete, consider the case where $g_B = dx^2$ and $g_F = dy^2$ and $M$ is $2$-dimensional Euclidean space with $\varphi= \varphi_2(y)$. Taking $U = \frac{\partial}{\partial y}$, write $X = v(t) \frac{\partial}{\partial x}$ then from (\ref{X_Ueqn})  we have 
\begin{align*}
  \frac{dv}{dt} &= e^{-\varphi_2(\sigma_2(t))} \frac{d \varphi_2}{dy} \frac{d \sigma_1}{dt} \\
  v(t) &= v(0) + \int_0^t \left( e^{-\varphi_2(\sigma_2(t))} \frac{d \varphi_2}{dy}|_{\sigma_2(t)} \frac{d \sigma_1}{dt} \right) dt.
\end{align*}
Let $\sigma$ be a square with vertices $(0, y_0), (a, y_0),(a, y_1), (0, y_1)$. Then we have $v(t) = v(0) + a \left( e^{-\varphi_2(y_0)} \frac{d \varphi_2}{dy}|_{y_0} - e^{-\varphi_2(y_1)} \frac{d \varphi_2}{dy}|_{y_1} \right). $ Which shows that, as long as the quantity $e^{-\varphi_2(y)} \frac{d \varphi_2}{dy}|_{y}$ is non-constant in $y$ that we have
\[
  \mathrm{Hol}^{\alpha}( \mathbb{R}^2)
  = \left\{ \begin{pmatrix}
    1 & c \\
    0 & 1
  \end{pmatrix} : c \in \mathbb{R} \right\},
\]
which is, of course isomorphic to $\mathbb{R}$. 

On the other hand, note that if $e^{-\varphi_2(y)} \frac{d \varphi_2}{dy}|_{y}$ is constant, then $\varphi_2= - \ln( d - cy)$ for some constants $c$ and $d$ and is thus not defined for all $y$. 
\end{example}

For a Levi-Civita connection, since the holonomy group is contained in the orthogonal group, if a distribution is preserved by the holonomy, then so is its orthogonal complement. Example~\ref{Ex:Warped} shows that this is not true generally for $\alpha$-holonomy.  Recall that a simply connected Riemannian manifold admits a parallel vector field if and only if it is isometric to a product metric with one factor a flat Euclidean space. The following example shows that this result is also not true for the $\alpha$-connection. 

\begin{example}
In the notation of Proposition~\ref{Prop:Twisted2}, assume that $B$ is one-dimensional, i.e. $g = dr^2 + e^{2\psi(r)} g_F$ and $\varphi = \psi(r) + \varphi_2$. Then, when $U = 0$, if we write $X = v(t) \frac{\partial}{\partial r}$, we have 
\[
  \frac{dv}{dt} = v \frac{d \psi}{dt},
\]
so that $v(t) = c e^{\psi(t)}$ and thus the field $X = k e^\varphi e^{\psi} \frac{\partial}{\partial r} = k e^{2\psi} e^{\varphi_2} \frac{\partial}{\partial r}$ is $\alpha$-parallel (along any curve). Thus in this case, we have that the elements of the holonomy group are of the form 
\[
  \begin{pmatrix}
    1 & C \\
    0 & h_2
  \end{pmatrix}.
\]
\end{example}
 
In fact, the following example shows that we can have an entire basis of $\alpha$-parallel fields which are not parallel. 
 
\begin{example}
Consider the previous example with $g_F$ Euclidean, $\varphi_2 = 0$, and $\psi(r) = r$. Then $(M,g)$ is hyperbolic space with constant curvature $-1$. From the previous example, we know that the field $e^{2r} \frac{\partial}{\partial r}$ is $\alpha$-parallel around any curve. 
 
Take $U = \frac{\partial}{\partial y_i}$ to be a parallel coordinate field on the Euclidean factor. Then the field $X_U$ is the solution to the equation 
\[
  \frac{dv}{dt} - v \frac{dr}{dt} + e^r \frac{dy_i}{dt} = 0 ,
\]
where $X_U = v(t) \frac{\partial}{\partial r}$. The general solution to this equation is $v(t) = (C+ y_i(t))e^{r(t)}$ We thus have that the field $\frac{\partial}{\partial y_i} + y_i e^{2r} \frac{\partial}{\partial r}$ is $\alpha$-parallel. In particular, we have a global basis of $\alpha$-parallel fields. Thus the holonomy group contains only the identity element. 
\end{example}
 
Now we consider an example where $\mathrm{Hol}^{\alpha}(M,g)$ is $\mathrm{SL}_n(\mathbb{R})$. 

\begin{example}\label{Ex:S2}
Consider the round 2-sphere $(S^2, dr^2 + \sin^2 r d\theta^2)$, with $\varphi =  \cos r$. Then, $\mathrm{Hol}^\alpha(S^2) = \mathrm{SL}_2(\mathbb{R})$. We prove this by considering the Lie algebra $\mathfrak{hol}^\alpha(S^2)$, and finding two families of loops, whose associated Lie algebra elements generate $\mathfrak{sl}_2(\mathbb{R})$. Both families will consist of loops based at a point on the equator: $p_0 = (\pi/2,0)$, written in the $(r,\theta)$ coordinates.

First consider a family of loops $\sigma^s(t)$ ($s\in[\pi/2,\pi)$) based at $p_0$. Each loop will consist of 3 smooth pieces:
\[
  \sigma^{s}(t) = \begin{cases}
    (t, 0) & t\in[\pi/2,s]\\
    (s, t-s) & t\in[s,s+2\pi]\\
    (2\pi+2s-t, 0) & t\in [2\pi+s,3\pi/2+2s].
  \end{cases}
\]
Note that $\sigma^{\pi/2}$ is just the equator. Let $h_s$ denote the element of $\mathrm{Hol}^\alpha(S^2)$ generated by $\sigma^s$. Then, explicit computation shows that
\[
  A = \left.\frac{d}{ds}h_s\right|_{s=\pi/2}
  = \begin{pmatrix}
    1 & -1\\
    2/\pi + 4\pi /3 & -1
  \end{pmatrix} \in \mathfrak{hol}^\alpha_{p_0}(S^2).
\]

The second family consists of loops
\[
  \sigma^s(t) = \begin{cases}
    (t,0) & t\in[\pi/2,\xi]\\
    (\xi, t-\xi) & t\in[\xi,\xi+s]\\
    (2\xi+s-t,s) & t\in[\xi+s,2\xi+s-\pi/2]\\
    (0,2\xi+2s-\pi/2-t) & t\in[2\xi+s-\pi/2,2\xi+2s-\pi/2].
  \end{cases}
\]
where $s\geq 0$, and $\xi = \cos^{-1}\left(\frac{1-\sqrt{5}}{2}\right)$ (chosen for convenience of computation). As before, let $h_s$ denote the holonomy element generated by $\sigma^s$. Then, we get:
\[
  B = \left.\frac{d}{ds}h_s\right|_{s=0}
  = \begin{pmatrix}
    0 & \frac{1-\sqrt{5}}{2} e^{\frac{\sqrt{5}-1}{2}}\\
    1 & 0
  \end{pmatrix}\in \mathfrak{hol}^\alpha_{p_0}(S^2).
\]

Since $[A,B]$ is linearly independent from $A$ and $B$, $A$ and $B$ generate a 3 dimensional Lie algebra, so $\mathfrak{hol}^\alpha_{p_0}(S^2) = \mathfrak{sl}_2(\mathbb{R})$, and $\mathrm{Hol}^\alpha(S^2) = \mathrm{SL}_2(\mathbb{R})$.
\end{example}

\begin{example}
We can follow the ideas of Example~\ref{Ex:S2} to show that for the round $n$-sphere $(S^n, dr^2 + \sin^2 r g_{S^{n-1}})$ with $\varphi = \cos r$, we get $\mathrm{Hol}^{\alpha}(S^n) = \mathrm{SL}_n(\mathbb{R})$.
\end{example}

\subsection{Parallel vector fields and $1$-forms}

The examples of the previous section show that splitting results for the Levi-Civita connections do not hold for $\nabla^\alpha$. In this section we show that there is still rigidity when we have $\alpha$-parallel vector fields and $1$-forms. 

In the case of an $\alpha$-parallel vector field we have the following classification. 

\begin{proposition} \label{Prop:f-ParallelField}
If there is an $\alpha$-parallel vector field, $V$, on a simply connected, complete manifold $M$, then $M$ is diffeomorphic to $\mathbb{R} \times N$, with a warped product metric $g_{M} = dr^2 + e^{2\psi(r)}g_N$ for some $\psi: \mathbb{R} \rightarrow \mathbb{R}$. Moreover, $\varphi = \psi(r) + \varphi_N(x)$ where $\varphi_N:N \rightarrow \mathbb{R}$ and $V = C e^{2\psi(r)}e^{\varphi_N(x)} \frac{\partial}{\partial r}$ for some constant $C$. 
\end{proposition}

\begin{remark}
Two special cases are:
\begin{enumerate}
\item If $\psi$ is constant, then we have the product metric $g = dr^2 + g_N$, $\varphi$ is a function on $N$, and $V = e^\varphi \frac{\partial}{\partial r}$. 

\item If $\varphi_N$ is constant then $\varphi = \psi$ and $V = e^{2\varphi} \frac{\partial}{\partial r}$. 
\end{enumerate}
\end{remark}

\begin{proof}
Suppose there is a non-zero vector field $V$ such that $\nabla_{\cdot}^\alpha V = 0$. Then for all $X$ we have 
\[
  \nabla_X V = d\varphi(X) V + d\varphi(V) X.
\]
Our first simple observation is that $V$ does not have a zero. This follows from the uniqueness of $\alpha$-parallel translation starting at a given point along with the fact that the zero vector field is always $\alpha$-parallel. 

Let $W = e^{-\varphi} V$, then we have 
\begin{equation} 
\label{Eqn:f-Parallel} \nabla_X W = e^{-\varphi} \left( -d\varphi(X) V + \nabla_X V \right) = e^{-\varphi}d\varphi(V) X = d\varphi(W) X . 
\end{equation}
In particular, this implies that $g(\nabla_X W, Y) = g(\nabla_Y W, X)$, which implies that the dual 1-form $\omega = g(W, \cdot)$ is a closed one-form. We also have that 
\[
  L_W g(X,Y) = g(\nabla_X W, Y) + g(\nabla_Y W , X) = 2 d\varphi(W)g(X,Y),
\]
so that $W$ is a closed conformal field. Such fields are classified in general. 

Since $W$ is a closed field and $M$ is simply connected we have that $W = \nabla u$ for some function $u:M \rightarrow \mathbb{R}$ and $\Hess u = d\varphi(\nabla u)g$. A result of Brinkmann \cite{Brinkmann} and Tashiro \cite{Tashiro} states that if $(M,g)$ is complete and supports a non-constant function $u$ such that $\Hess u = \chi g$ for some function $\chi$, then $g=dr^2 + \rho^2(r) g_N$, $u = u(r)$, $\frac{du}{dr} = \rho$.

Since $V$ is never zero, $u$ does not have any critical points so we also have that $\psi$ is never zero and we have a global topological splitting $M= \mathbb{R} \times N$. We also have that 
\begin{equation}
  \frac{d^2 u}{dr^2} = \Hess u \left( \frac{\partial}{\partial r},\frac{\partial}{\partial r}\right) = g(\nabla \varphi, \nabla u)= \frac{\partial \varphi}{\partial r} \frac{du}{dr}. \label{Eqn:WarpingFunc}
\end{equation}
Since $\frac{du}{dr}$ is never zero this implies that $\frac{\partial \varphi}{\partial r}$ is a function of $r$. Then, for any field $U$ on $N$, since $[ U, \frac{\partial}{\partial r}] = 0$, we have 
\[
  0 = D_U D_{\frac{\partial}{\partial r}}(\varphi) = D_{\frac{\partial}{\partial r}} \left(D_U \varphi \right).
\]
This implies that $D_U \varphi$ is a function on $N$. Then we can write $\varphi(r,x) = \psi(r) + \varphi_N(x)$ where $\psi: \mathbb{R} \rightarrow \mathbb{R}$ and $\varphi_N: N \rightarrow \mathbb{R}$. Then \eqref{Eqn:WarpingFunc} becomes $\rho'(r) = \psi'(r)\rho(r)$ so that $\rho = Ce^{\psi}$ for some constant $C$. 

Putting this all together, we have that the metric splits as a warped product $g = dr^2 + e^{2\psi(r)}g_N$ where $\varphi(r,x)= \psi(r) + \varphi_N(x)$ and 
\[
  V = e^{\varphi} W = e^\varphi \nabla u = Ce^\varphi e^{\psi} \frac{\partial}{\partial r} = C e^{2\psi(r)}e^{\varphi_N(x)} \frac{\partial}{\partial r}.
\]
\end{proof}

\begin{remark}
Note that in the case where $M$ is not simply connected, the proposition can be applied to the universal cover with pullback density. 
\end{remark}

Now we consider the case of linearly independent $\alpha$-parallel fields.

\begin{theorem}\label{TopParFields}
Suppose that $(M,g,f)$ is complete, simply connected and  admits $m$ linearly independent $\alpha$-parallel vector fields. Then  $M$ is diffeomorphic to $\mathbb{R}^m \times L$ for some $n-m$ dimensional manifold $L$, $g$ is a warped product metric
\[
g = h +  e^{2\varphi} g_L
\]
with $\varphi = \overline{\psi} + \varphi_L$, where  $h$ is a Euclidean or hyperbolic metric on  $\mathbb{R}^m$ and $\overline{\psi}: \mathbb{R}^m \rightarrow \mathbb{R}$. 
\end{theorem}

\begin{proof}
We proceed by induction on $m$.  The case $m=1$ is handled by  Proposition~\ref{Prop:f-ParallelField}.  Now suppose that $(M,g,f)$ has $m$-linearly independent fields.  Let $P$ be an $\alpha$-parallel field and complete $P$ to a basis for the space parallel fields, $\{P, Q_1, Q_2, \dots, Q_{m-1} \}$.  

By Proposition~\ref{Prop:f-ParallelField} we can write the metric $g = dr^2 + e^{2\psi(r)} g_N$ with $\varphi = \psi + \varphi_N$ and $P = C e^{2\psi} e^{\varphi_N} \frac{\partial}{\partial r}$. 

Write $Q_i = a_i(r,x) \frac{\partial}{\partial r} + \sum_{j=1}^{n-1}b_{ij}(r,x)  E_j$ where $E_j$ form a local basis of vector fields on $N$. Then
\begin{align} \label{TopPar1}
  \nabla^\alpha_{\frac{\partial}{\partial r}} Q_i &= \frac{\partial a_i}{\partial r}\frac{\partial}{\partial r} + \sum_{j=1}^{n-1} \left(\frac{\partial b_{ij}}{\partial r}E_j + b_{ij} \frac{\partial \psi}{\partial r}E_j \right) - d\varphi\left( \frac{\partial}{\partial r} \right) Q_i - d\varphi(Q_i) \frac{\partial}{\partial r}  \nonumber\\
  &= \left( \frac{\partial a_i}{\partial r} - 2 a_i \frac{\partial \psi}{\partial r} -  \sum_{j=1}^{n-1} b_{ij} D_{E_j} \varphi_N \right) \frac{\partial}{\partial r} +  \sum_{j=1}^{n-1} \frac{\partial b_{ij}}{\partial r} E_j.
\end{align}
Setting this equal to zero, since the $E_j$ are linearly independent, shows  that all of the $b_{ij}$ are constant in the $r$ direction. Thus we can write $Q_i = a_i(r,x) \frac{\partial}{\partial r} + U_i$ where $U_i$ is a field on $N$. 

Now let $V$ be a field on $N$, then 
\begin{align} \label{TopPar2}
  \nabla^\alpha_{V} Q_i &= (D_V a_i) \frac{\partial}{\partial r} + a_i \frac{\partial \psi}{\partial r} V + \nabla^N_V U_i - e^{2\psi} \frac{\partial \psi}{\partial r}g_N(U_i,V) \frac{\partial}{\partial r}- d\varphi(V) Q - d\varphi(Q) V  \nonumber\\
  &= \left( D_V a_i - (D_V \varphi) a_i - e^{2\psi} \frac{\partial \psi}{\partial r}g_N(U_i,V) \right) \frac{\partial}{\partial r} + \nabla_V^{\varphi_N} U_i. 
\end{align}
Therefore, for $Q_i$ to be parallel we need $U_i$ to be an $\alpha_N$-parallel field on $(N, g_N)$, with $\alpha_N = d\varphi_N$.  

Then  the set $\{U_i\}_{i=1}^{m-1}$ is a set of $m-1$ linearly independent $\alpha_N$-parallel fields on $g_N$.  Applying the induction hypothesis gives us that $N$ is diffeomorphic to $\mathbb{R}^{m-1} \times L$,  $g_N = h' + e^{2\varphi_N} g_{L}$, $\varphi_N = \rho + \varphi_L$.  Then we have $M = \mathbb{R} \times N = \mathbb{R}^m \times L$, $\varphi  = \psi + \rho + \varphi_L = \overline{\psi} + \varphi_L$  and 
\begin{align*}
  g &= dr^2 +  e^{2\psi(r)}\left( h' + e^{2\varphi_N} g_L\right) \\
  &= dr^2 +e^{2\psi(r)} h' + e^{2\varphi}g_L \\
  &= h +e^{2\varphi} g_L,
\end{align*}
where $h =  dr^2 +e^{2\psi(r)} h'$ and  $\overline{\psi} = \psi+ \rho$.  

Now we wish to show that $h$ is constant curvature and thus a flat or hyperbolic metric.   We consider the cases $m=2$ and $m>2$ separately. 

First assume that $m>2$.  Note that the $m$ linearly independent $\alpha$-parallel fields are all  tangent to the $\mathbb{R}^m$ factor. The fields  themselves might not be fields on $\mathbb{R}^m$, however, for any parallel field $P$ on $M$ one has 
\[
  R^{\alpha}(X,Y)P = \nabla^{\alpha}_X\left(\nabla^{\alpha}_YP\right) - \nabla^{\alpha}_Y\left(\nabla^{\alpha}_X P\right) -  \nabla^{\alpha}_{[X,Y]}P = 0.
\]
Since $R^{\alpha}(X,Y)Z$ is a tensor, this shows that for any vectors $v, w, z$ that are tangent to the $\mathbb{R}^m$ factor,  $R^{\alpha}(v,w) z = 0$.  The fact that  $\varphi = \overline{\psi} + \varphi_L$,  combined with  equation  \eqref{twistprod1}, shows that $R^{\alpha}(v,w)z = R^{h, d\overline{\psi}}(v,w)z$, where $R^{h, d\overline{\psi}}$ denotes the curvature of the connection $\nabla^{h, d\overline{\psi}}$ on $\mathbb{R}^m$.  

This implies that  the  connection $\nabla^{h, d\overline{\psi}}$ is flat. In particular, 
\[
  0 = h(R^{d\overline{\psi}}(v,w)w,v) = \sec(v,w) + \Hess \overline{\psi}(w,w) + d\overline{\psi}(w)^2.
\]
Since $\sec(v,w) = \sec(w,v)$, for any two vectors $v,w$ at a fixed point $p$, $\Hess \overline{\psi}(v,v) + d\overline{\psi}(v)^2 = \Hess \overline{\psi}(w,w) + d\overline{\psi}(w)^2$.  But this implies that $\sec_p = c(p)$ for some function $c$.  By Schur's Lemma we have that $(\mathbb{R}^m, h)$ is constant curvature, and thus must be a hyperbolic or Euclidean metric. Also see Proposition 1.1 of \cite{WylieSec}. 

When $m=2$  we can show directly that the metric $(\mathbb{R}^2, h)$ has constant curvature using equations \eqref{TopPar1} and \eqref{TopPar2}. Let $(M,g,f)$ have two linearly independent $\alpha$-parallel fields, $P$ and $Q$.  Then we have that the metric is of the form
\begin{align*}
  g &= dr^2 + e^{2 \psi(r)}\left( ds^2 + e^{2\rho(s)}g_L \right) \\
  &=  (dr^2 + e^{2 \psi(r)}ds^2) + e^{2 \psi(r)}e^{2\rho(s)}g_L\\
  &= h + e^{2 \psi(r)}e^{2\rho(s)}g_L,
\end{align*}
where $\varphi = \psi(r) + \rho(s) + \varphi_L$ and $P = Ce^{2\psi(r)}e^{\rho(s)}e^{\varphi_L} \frac{\partial}{\partial r}$ and $Q = a \frac{\partial}{\partial r} + B e^{2\rho(s)}e^{\varphi_L} \frac{\partial}{\partial s}$ and $B$ and $C$ are constants.  By rescaling $P$ and $Q$, we can assume that $B=C=1$.

Equations \eqref{TopPar1} and \eqref{TopPar2} then imply that we have the equations
\begin{align}
\label{2fields1} \frac{\partial a}{\partial r} - 2a \frac{d\psi}{dr} &= e^{2\rho(s)}e^{\varphi_L} \frac{d\rho}{d s} \\
\label{2fields2} \frac{\partial a}{\partial s}- a \frac{d \rho}{d s} &= e^{2\psi(r)}e^{2\rho(s)} e^{\varphi_L} \frac{d \psi}{d r}\\
\nonumber  D_V a - (D_V \varphi_L) a &= 0,
\end{align}
where $V$ is any field on $L$. 

Differentiating \eqref{2fields1} with respect to $s$ and \eqref{2fields2} with respect to $r$ and then subtracting the result gives the equation
\[
 -2 \frac{\partial a}{\partial s} \frac{d \psi}{d r} +  \frac{\partial a}{\partial r} \frac{d \rho}{ds} = e^{\varphi_L} \left(\frac{d}{ds}\left( e^{2\rho(s)} \frac{d \rho}{d s} \right) - e^{2\rho}\frac{d}{dr} \left( e^{2\psi(r)} \frac{d \psi}{d r} \right) \right).
\]
Plugging in \eqref{2fields1} and \eqref{2fields2}  on the left hand side and expanding the derivatives on the right gives the equation
\[
  \frac{d^2 \psi}{d r^2}  e^{2\psi} = \frac{d ^2\rho}{d s^2} + \left( \frac{d\rho}{ds} \right)^2. 
\]
Since the left hand side depends only on $r$ and the right hand side only on $s$, we get that these quantities must be constant, call the constant $j$.

Let $u(r) = \int_0^r e^{-2\psi(t)} dt$ and let $'$ denote derivative with respect to $r$.  Then, integrating  the equation $\psi'' e^{2\psi}=j$ gives us 
\[
  \psi' = j u + k
\]
for a constant $k$.  Moreover,  $\ln\left(u'\right) = -2\psi(r)$ implies that $\psi' =\frac12 \frac{u''}{u'} $ so in terms of $u$ we have

\begin{align}
  -\frac12 \frac{u''}{u'} &= j u + k  \nonumber\\
  u'' &= -2ju u' -2k u'  \nonumber\\
  u' &= -j u^2 -2ku + l,  \label{ueqn}
\end{align}
where $l$ is another (positive) constant. 

Computing the sectional curvature for the metric $h$,  $\sec\left( \frac{d}{dr}, \frac{d}{ds}\right)$, we obtain
\begin{align*}
  \sec\left( \frac{d}{dr}, \frac{d}{ds}\right) &= - \frac{ \frac{d^2}{dr^2} e^{\psi}}{e^{\psi}} \\
  &= -\frac{d^2 \psi}{dr^2} - \left( \frac{d\psi}{dr}\right)^2 \\
  &= -j u' - (ju+k)^2 \\
  &= -(k^2 + jl).
\end{align*}

Hence, $\sec$ is constant as claimed.
\end{proof}

\begin{remark}
The simplest non-flat example in Theorem~\ref{TopParFields} comes from taking $j=0$, $\rho = 0$ and  $\psi(r) = kr$, then we have an example of the form
\[
g = dr^2 + e^{2kr}(ds_1^2 + \cdots +  ds^2_{m-1}) + e^{2kr} g_L,
\]
where $h=dr^2 + e^{2kr}(ds_1^2 + \cdots +  ds^2_{m-1}) $ is the hyperbolic metric of curvature $-k^2$. 
\end{remark}

\begin{remark}
Theorem~\ref{TopParFields} together with Proposition~\ref{Prop:f-ParallelField} yield Theorem~\ref{Thm:Parallel}.
\end{remark}

Since the Riemannian metric $g$ is not parallel with respect to $\nabla^\alpha$, the dual $1$-form to a parallel vector field will not be parallel. However, we do obtain the following relation between them:

\begin{proposition}
A 1-form $\omega = e^{-\varphi}g(X,\cdot)$ is $\alpha$-parallel on $(M,g,f)$ iff $X$ is $(-\alpha)$-parallel on $(M,e^{-2\varphi}g,-f)$.
\end{proposition}

\begin{proof}
First consider
\begin{align*}
  \left(\nabla_U^\alpha \omega\right)(V) &= D_U \omega(V) - \omega\left(\nabla_U^\alpha V\right)\\
  &= D_U \left[ e^{-\varphi} g(X, V) \right] - e^{-\varphi} g\left(X, \nabla_U^\alpha V\right)\\
  &= -d\varphi(U) e^{-\varphi}g(X,V) + e^{-\varphi} D_U g(X, V) - e^{-\varphi} g\left(X,\nabla_U V\right)\\
  &\qquad\qquad+ d\varphi(U) e^{-\varphi} g(X, V) + d\varphi(V) e^{-\varphi} g(X, U)\\
  &= e^{-\varphi} g\left(\nabla_U X, V\right) + d\varphi(V) e^{-\varphi} g(X, U).
\end{align*}

Next, let $\widetilde{\nabla}$ denote the Levi-Civita connection of $(M,e^{-2\varphi}g)$, and $\widetilde{\nabla}^{-\alpha}$ the connection for the manifold with density $(M, e^{-2\varphi}g, -f)$. Then,
\begin{align*}
  e^{-\varphi}g\left(\widetilde{\nabla}_U^{-\alpha} X, V\right) &= e^{-\varphi} g\left(\widetilde{\nabla}_U X, V\right) + d\varphi(U) e^{-\varphi} g(X, V) + d\varphi(X) e^{-\varphi} g(U, V)\\
  &= e^{-\varphi} g\left(\nabla_U X, V\right) - d\varphi(U) e^{-\varphi} g(X, V) - d\varphi(X) e^{-\varphi} g(U, V)\\
  &\qquad\qquad+ d\varphi(V) e^{-\varphi} g(X,U) + d\varphi(U) e^{-\varphi} g(X, V)\\
  &\qquad\qquad+ d\varphi(X) e^{-\varphi} g(U, V)\\
  &= e^{-\varphi} g\left(\nabla_U X, V\right) + d\varphi(V) e^{-\varphi} g(X, U)\\
  &= \left(\nabla_U^\alpha \omega\right)(V).
\end{align*}
\end{proof}

\subsection{The de Rham splitting theorem for weighted Riemannian manifolds} 

We now consider the more general question of when we have a distribution which is $\alpha$-parallel translation invariant.  Let  $\nu$ be an $\alpha$-parallel translation invariant distribution of a manifold $(M,g)$, where $\alpha=d\varphi$.  Let  $\nu^{\perp}$ be the orthogonal complement of $\nu$.  Our first result is the following Lemma. 

\begin{lemma}
$\nu$ and $\nu^{\perp}$ are both integrable distributions and $\nu$ is totally geodesic. 
\end{lemma}

\begin{proof}
Since $\nabla^\alpha$ is a torsion free connection, $\nu$ is integrable. For $X, Y, Z$ vector fields on $M$ we have
\begin{align}
\nonumber e^{-\varphi} g( \nabla_X^\alpha Y, Z ) &= e^{-\varphi} g(\nabla_XY, Z) - d\varphi(X) e^{-\varphi} g(Y,Z) - d\varphi(Y) e^{-\varphi} g(X,Z) \\
\label{de Rham Eqn} &= g( \nabla_X(e^{-\varphi} Y), Z) - d\varphi(Y) e^{-\varphi} g(X,Z) \\
\nonumber &= D_X (e^{-\varphi}g(Y,Z)) - e^{-\varphi} g( Y, \nabla_X Z) - d\varphi(Y) e^{-\varphi} g(X,Z) .
\end{align}

First choose $X, Y$ fields in $\nu$ and $Z \in \nu^{\perp}$, then \eqref{de Rham Eqn} gives that $g( Y, \nabla_X Z) = 0$, implying that $\nu$ is totally geodesic.

On the other hand if we apply \eqref{de Rham Eqn} to $X,Y \in \nu^{\perp}$ and $Z \in \nu$ we have 
\begin{align*}
  g(\nabla^\alpha_X Y, Z) &= - g(Y, \nabla_X Z) \\
  &= -d\varphi(Z) g(X, Y).
\end{align*}
So, $g([X,Y], Z) = g(\nabla^\alpha_X Y, Z) - g(\nabla^\alpha_Y X, Z) = 0$, and thus $\nu^{\perp}$ is also integrable.
\end{proof}

Since $\nu$ and $\nu^{\perp}$ are both integrable we have  local topological product structure in a coordinate neighborhood $U$ which is diffeomorphic to $B \times F$ where $B \times \{x\}$ are the integral submanifolds of $\nu$ and $\{p\} \times F$ are the integral submanifolds of $\nu^{\perp}$.  In these coordinates, any vector can be written uniquely as $X = Y+Z$ where $Y\in \nu$ and $Z\in \nu^{\perp}$.  Define a new perturbation of the metric $g$ by the formula 
\[
k(X_1, X_2) = g(Y_1, Y_2) + e^{-2\varphi}g(Z_1, Z_2),
\]
where $X_i = Y_i+Z_i$, $Y_i \in \nu_p$ and $Z_i \in \nu^{\perp}_p$.  $k$ is a smooth metric on $M$. Then we have the following lemma. 

\begin{lemma}
$\nu$ is invariant with respect to the Riemannian connection $\nabla^{k}$. 
\end{lemma}

\begin{proof}
Let $Y$ be a field in $\nu$ and $Z \in \nu^{\perp}$.  Then, by the Kozul formula, for any vector field $X$,  
\begin{align}
\label{Kozul} 2 k(\nabla^k_X Y, Z) &= D_Y\left(k(Z,X)\right) - D_Z \left(k(X,Y) \right) \\
\nonumber &\qquad + k\left([X,Y],Z\right) - k\left([Y,Z],X\right) + k\left([Z,X],Y\right).
\end{align}
Letting $X \in \nu$, and using that $\nu$ is totally geodesic and integrable, gives
\begin{align*}
  2 k\left(\nabla^k_X Y, Z\right) &= - D_Z \left(k(X,Y) \right) - k\left([Y,Z],X\right) + k\left([Z,X],Y\right)\\
  &= - D_Z \left(g(X,Y) \right) - g\left([Y,Z],X\right) + g\left([Z,X],Y\right)\\
  &= 2g\left(\nabla^g_X Y, Z\right) \\
  &= 0.
\end{align*}
On the other hand, if we let $X \in \nu^{\perp}$ in \eqref{Kozul} and use that $\nu^{\perp}$ is integrable we obtain
\begin{align*}
  2 k(\nabla^k_X Y, Z) &= D_Y\left(k(Z,X)\right) + k\left([X,Y],Z\right) - k\left([Y,Z],X\right)  \\
  &=  D_Y\left(e^{-2\varphi}g(Z,X)\right) + e^{-2\varphi}g\left([X,Y],Z\right) - e^{-2\varphi}g\left([Y,Z],X\right) \\
  &= e^{-2\varphi}\left(  -2 d\varphi(Y) g(Z,X) +D_Y(g(X,Z)) +  g\left([X,Y],Z\right) -  g\left([Y,Z],X\right)\right)\\
  &= 2e^{-2\varphi}\left( - g(Z, d\varphi(Y)X) + g(\nabla^g_X Y, Z) \right) \\
  &= 2e^{-2\varphi}\left( g(\nabla^{g, \alpha}_X Y, Z) \right)\\
  &= 0.
\end{align*}
Therefore, $\nabla^k_X Y \in \nu$ for all $X$.
\end{proof}

This gives us the following local version of the de Rham decomposition theorem.
 
\begin{theorem}\label{Thm:Local deRham}
Let $(M,g)$ be a Riemannian manifold and $\alpha$ a closed one-form on $M$.  Let $\nu$ be an $\alpha$-parallel translation invariant distribution, then $M$ is locally a twisted product.   That is all $p \in M$ admit an open neighborhood $U_p$ such that $U$ is diffeomorphic to $B \times F$, where $\nu$ is the tangent distribution to $B \times \{ x\}$, and $g = g_B + e^{2\varphi} g_F$ for fixed metrics $g_B$ and $g_F$.  
\end{theorem}

\begin{proof}
From the Lemma we have that $\nu$ is an invariant distribution for the Riemannian metric $k$. Applying the local version of the Riemannian de Rham decomposition theorem \cite{Besse}*{10.38} then gives that $k$ must locally be a product 
\[
k = g_B + g_F
\]
for fixed metric $g_B$ and $g_F$.  This then gives that $g = g_B + e^{2\varphi} g_F$.
\end{proof}

For the global splitting we also have the following result which follows directly from the global de Rham decomposition theorem in the Riemannian case \cite{Besse}*{10.43}. 

\begin{theorem}\label{Thm:GlobaldeRham}
Let $(M,g)$ be a simply connected Riemannian manifold and $\alpha$ a closed one form on $M$.  Let $\nu$ be an $\alpha$-parallel translation invariant distribution.  If the metric $k$ is complete then $M$ is a twisted product metric,  that is $M$ is  diffeomorphic to $B \times F$, where $\nu$ is the tangent distribution to $B \times \{ x\}$, and $g = g_B + e^{2\varphi} g_F$ for fixed metrics $g_B$ and $g_F$.  
\end{theorem} 

It seems somewhat unclear what is the exact relationship between the completeness of the metric $g$, completeness of the connection $\nabla^{\alpha}$, and completeness of the metric $k$.  However, we do have the following proposition giving sufficient conditions for $k$ completeness. 

\begin{proposition} \label{Prop:kcomplete}
Let $(M,g)$ be a complete simply connected Riemannian manifold and $\alpha=d\varphi$ a smooth closed one form supporting an $\alpha$-parallel translation invariant distribution.  If 
\begin{enumerate}
\item $\varphi$ is bounded from above, or

\item $\nabla^{\alpha}$ is  geodesically complete, and $\nabla \varphi(p) \subset \nu(p)$,  $\forall p \in M$
\end{enumerate}
then $k$ is complete. 
\end{proposition}

\begin{remark}
In several of the comparison results we were able to replace the a priori assumptions of $f$ being bounded from above by the weaker condition of $\alpha$-completeness. As such, it is natural to ask whether the same weakening of hypotheses can be done in this case as well.
\end{remark}

\begin{proof}
To prove part (1), if $\varphi \leq A$ then 
\begin{align*}
  k(X_1, X_2) &= g(Y_1, Y_2) + e^{-2\varphi} g(Z_1, Z_2)\\
 &\geq g(Y_1, Y_2) + e^{-2A} g(Z_1, Z_2)\\
 &\geq B g(X_1, X_2),
\end{align*}
where $B= \min\{1, e^{-2A}\}$.  Therefore, for any curve $\sigma$, $\mathrm{length}^k (\sigma) \geq \left(\sqrt{B}\right) \mathrm{length}^g(\sigma)$, so a bounded subset in the $k$ metric must also be bounded in the $g$ metric.  Since $g$ is complete, the closed and bounded subsets with respect to $g$ must be compact.  Therefore, the same is true for $k$ and $k$ is complete by the Hopf-Rinow theorem. 
 
The proof of (2) is a little more subtle. Let  $\sigma$ be a $k$-geodesic and let $p=\sigma(t)$ for some $t$.  In a neighborhood of $p$, by Theorem~\ref{Thm:Local deRham} we can write $k$ as a product  $k=g_B+g_F$, so we can write $\sigma(t) = (\sigma_1(t), \sigma_2(t))$ where $\sigma_1$ is a geodesic in $g_B$ and $\sigma_2$ is a geodesic in $g_F$.  The geodesics in $g_B$ can always be extended because $B$ is an integral submanifold of $\nu$ which is totally geodesic distribution  in the complete metric $g$.  
 
To see that the geodesics in $g_F$ can always be extended, note that we have the local splitting $g = g_B + e^{-2\varphi}g_F$, given by Theorem~\ref{Thm:Local deRham}.  The fact that $\nabla \varphi(p) \subset \nu(p)$ tells us that $\varphi$ is a function of $B$ only and thus the metric is locally a warped product. We thus satisfy  all of the hypotheses of  Proposition~\ref{Prop:Twisted2} with  $\alpha_2 = 0$.  According to Proposition~\ref{Prop:Twisted2}, $\alpha$-parallel translation  in $g$ acts on vectors tangent to $F$ by parallel translation in $g_F$.  This implies that the $\alpha$-geodesics in $g$ are of the form $\eta(s) = (\eta_1(s), \eta_2(s))$, where $\eta_2(s)$ is a geodesic in $g_F$.  Therefore, $\alpha$-completeness implies that the geodesics of $g_F$ can be extended for all time.  
 
Therefore, the geodesic $\sigma$ can be extended for all time and thus $k$ is complete. 
\end{proof}

\begin{remark}
(1) of Proposition~\ref{Prop:kcomplete} shows that if $(M,g)$ is compact and has a $\nabla^{\alpha}$-invariant distribution, then the universal cover admits a global splitting. Thus Theorems ~\ref{Thm:Local deRham} and ~\ref{Thm:GlobaldeRham} along with Proposition~\ref{Prop:kcomplete} prove Theorem \ref{Thm:deRham}.
\end{remark}

\subsection{Compact $\alpha$-Holonomy}

An important question is what natural parallel tensors (if any) does $\nabla^\alpha$ have. We consider the special case of  when $\nabla^{\alpha}$ admits a parallel  two-tensor.  From Proposition~\ref{Prop:FundPrinc1} it follows that admitting an invariant positive definite two-tensor is equivalent to $\mathrm{Hol}^{\alpha}(g)$ being compact.   To see this note that if $\mathrm{Hol}^{\alpha}(g)$ is compact, then, by averaging over the group, one can  find a positive definite metric at a point $p$  which is invariant under $\mathrm{Hol}^{\alpha}(g)$.  Conversely, if $\nabla^{\alpha}$ admits a parallel metric then $\mathrm{Hol}^{\alpha}(g)$ must be compact as it is a closed subgroup of $O(n)$. Note that, in particular, this implies that $\nabla^{\alpha}$ is the Levi-Civita connection of some Riemannian metric if and only if $\mathrm{Hol}^{\alpha}(g)$ is compact. 

A diffeomorphism of  manifolds  equipped with  affine connections is called an \emph{affine transformation} if it preserves the connections.  More generally, a diffeomorphism is called \emph{projective transformation} or a \emph{geodesic mapping} if it maps un-paremetrized geodesics to geodesics.  The following propostition follows from Weyl's theorem. 

\begin{proposition}
Let $(M,g)$ be a Riemannian manifold. There exists a non-zero one-form $\alpha$ such that $\mathrm{Hol}^{\alpha}(M,g)$ is compact if and only if there is a non-affine geodesic mapping from $(M,g)$ to another Riemannian manifold $(N,\widetilde{g})$. 
\end{proposition}

\begin{remark}
Note that even if $g$ is complete,  $\widetilde{g}$ need not be complete. However, we assume $(M,g, \alpha)$ is  $\alpha$-complete, then $\widetilde{g}$ is a complete metric.
\end{remark}

\begin{remark} 
If $\mathrm{Hol}^{\alpha}(M,g) = G$ is a proper subgroup of $O(n)$, it also follows that $\widetilde{g}$ has special holonomy $G$. 
\end{remark}

Understanding geodesic mappings between Riemannian manifolds  is an old problem going back to Beltrami and Levi-Civita, see for example the survey article \cite{Mikes}.  There are various rigidity theorems about when two Riemannian spaces can be related by a geodesic mapping.  For example, Beltrami showed that if there is a geodesic mapping from $(M,g)$ to a Euclidean space, then $(M,g)$ is a space of constant curvature.  Note that this is equivalent to the special case $m=n$ in Theorem~\ref{TopParFields}.

Another basic known fact about geodesic mappings is that there are no projective transformations from  $(M,g)$  to itself aside from affine transformations.  In terms of the $\alpha$-holonomy this is equivalent to saying that $g$  never parallel with respect to $\nabla^\alpha$ unless $\alpha$ is trivial. For completeness we include a proof. 

\begin{proposition}
Let $(M,g)$ be a manifold. $\nabla^\alpha g = 0$ if and only if $\alpha\equiv 0$. 
\end{proposition}

\begin{proof}
Using the fact that $\nabla g = 0$ we have
\[
(\nabla_Z^\alpha g)(X,Y) = 2 \alpha(Z) g(X,Y) + \alpha(X)g(Z,Y) + \alpha(Y) g(X, Z). 
\]
So we have $0 = 2 \alpha(Z) g(X,Y) + \alpha(X)g(Z,Y) + \alpha(Y) g(X, Z)$. However, if we take $X=Y$ and $Z$ perpendicular to $X,Y$ we obtain $\alpha(Z) = 0$. Since we can choose $Z$ arbitrarily, this shows that $\alpha\equiv 0$. 
\end{proof}

On the other hand, from the equation above we, if we let $h = e^{-\varphi} g$ we have that 
\[
(\nabla_Z^\alpha h)(X,Y) = d\varphi(Z) h(X,Y) + d\varphi(X)h(Z,Y) + d\varphi(Y) h(X, Z).
\]
In the literature on geodesic mappings, this equation and  the formula for the $\nabla^{\alpha}$ connection are sometimes called  the \emph{Levi-Civita equations} \cite{Mikes}. 

 In particular, $(\nabla_Z^\alpha h)(X,Y) = (\nabla_Y^\alpha h)(X,Z)$. In Riemannian geometry a symmetric two-tensor $b$ that satisfies $(\nabla_Z b)(X,Y) = (\nabla_Y b)(X,Z)$ is called a Codazzi tensor. We have the following relationship between Codazzi tensors for $(M,g)$ and $\nabla^{\alpha}$, 

\begin{proposition}
Let $(M,g)$ be a manifold with density, and let $T$ be a symmetric 2-tensor. Then, $T$ is a Codazzi tensor of $(M,g)$ iff $e^{-\varphi} T$ is a Codazzi tensor with respect to the connection $\nabla^\alpha$.
\end{proposition}

\begin{proof}
Consider
\begin{align*}
  e^{\varphi}\left(\nabla^\alpha_X e^{-\varphi} T\right)(Y,Z) &= D_X T(Y,Z) - d\varphi(X) T(Y,Z) - T(\nabla_X^\alpha Y, Z) - T(Y, \nabla_X^\alpha Z)\\
  &= D_X T(Y,Z) - d\varphi(X) T(Y,Z) - T(\nabla_X Y, Z)\\
  &\qquad + d\varphi(X) T(Y,Z) + d\varphi(Y) T(X, Z) - T(Y, \nabla_X Z)\\
  &\qquad + d\varphi(X) T(Y,Z) + d\varphi(Z) T(Y,X)\\
  &= \left(\nabla_X T\right)(Y,Z) + d\varphi(X) T(Y,Z) + d\varphi(Y) T(X,Z)\\
  &\qquad + d\varphi(Z) T(Y,X).
\end{align*}

Therefore, $e^{-\varphi}T$ is Codazzi with respect to $\nabla^\alpha$ iff $T$ is Codazzi.
\end{proof}

One of the most common types of Codazzi tensor is a parallel 2-tensor. Which leads us to the following Corollary:

\begin{corollary}
Let $(M,g)$ be a manifold with density with an $\alpha$-parallel symmetric 2-tensor $T$, then $e^\varphi T$ is a Codazzi tensor on $(M,g)$.

In particular, if $\mathrm{Hol}^\alpha(M,g)$ is compact, then $(M,g)$ admits a positive definite Codazzi tensor. More generally, if $\mathrm{Hol}^\alpha(M,g)\subset O(p,q)$, then $(M,g)$ must admit a Codazzi tensor of signature $(p,q)$.
\end{corollary}

For more general discussion on Codazzi tensors see \cites{Besse, Codazzi}.  It does not seem that the connection between Codazzi tensors and geodesic mappings has been emphasized in  the literature. Note that  \cite{DS} provides some useful topological consequences of the existence of Codazzi tensors, especially in dimension 4.

The assumption that $\mathrm{Hol}^\alpha$ is compact has some interesting consequences in the case of positive curvature.

\begin{reptheorem}{Thm:CptCurv}
Let $(M,g,f)$ be a manifold with density, with $\mathrm{Hol}^\alpha(M,g)$ compact. Let $\widetilde{g}$ be a metric compatible with $\nabla^\alpha$. Then $\overline{\sec}_\varphi > 0$ implies that $\sec_{\widetilde{g}} > 0$.

Furthermore if $\Ric^1_f > 0$ then $\Ric_{\widetilde{g}}>0$.

Similarly, both parts also hold for non-negative, negative and non-positive curvature.
\end{reptheorem}

\begin{remark}
Similarly, this result also works for quasi-positive and almost positive curvature.
\end{remark}

\begin{proof}
Fix a point $p\in M$ and a vector $Y\in T_p M$ with $g(Y,Y)=0$. We will show that $\overline{\sec}_\varphi^Y > 0$ implies that $\sec_{\widetilde{g}}(\cdot,Y) > 0$.

Recall that $\overline{\sec}_f^Y(X) = g(R^\alpha(X,Y)Y,X)$. Also, since $\widetilde{g}$ is compatible with $\nabla^\alpha$, we know that $\sec_{\widetilde{g}}(X,Y) = \widetilde{g}(R^\alpha(X,Y)Y,X)$. Therefore, $R^\alpha(\cdot,Y)Y:T_p M \to T_p M$ is a symmetric operator with respect to $\widetilde{g}$.

Consider a basis on $T_p M$ consisting of eigenvectors of $R^\alpha(\cdot,Y)Y$: $X_1,\ldots,X_n$. Note that $R^\alpha(Y,Y)Y = 0$, so $Y$ is an eigenvector, call it $X_1$, then $\lambda_1 = 0$. Order the remaining eigenvalues $\lambda_2 \leq \lambda_3 \leq \cdots \leq \lambda_n$.

Let $\widetilde{X_k} = a X_k + b Y$ satisfying $g(\widetilde{X_k},\widetilde{X_k}) = 1$ and $g(\widetilde{X_k},Y) = 0$. Then,
\begin{align*}
  \overline{\sec}_\varphi^Y(\widetilde{X_k}) &= g(R^\alpha(\widetilde{X_k},Y)Y,\widetilde{X_k})\\
  &= g(R^\alpha(a X_k + bY, Y)Y, \widetilde{X_k})\\
  &= g(R^\alpha(a X_k, Y)Y, \widetilde{X_k}) + g(R^\alpha(bY, Y)Y, \widetilde{X_k})\\
  &= g(a \lambda_k X_k, \widetilde{X_k})\\
  &= \lambda_k g(\widetilde{X_k} - bY, \widetilde{X_k})\\
  &= \lambda_k,
\end{align*}
so, if $\overline{\sec}_\varphi^Y > 0$, all $\lambda_k > 0$, similarly for $\overline{\sec}_\varphi^Y \geq 0$, $\overline{\sec}_\varphi^Y < 0$ and $\overline{\sec}_\varphi^Y \leq 0$.

The Ricci part of the theorem follows immediately from the fact that $\Ric_f^1 = \Ric^\alpha = \Ric_{\tilde{g}}$.
\end{proof}


\end{document}